\documentclass[preprint,12pt]{elsarticle}

\usepackage{enumitem}
\usepackage{hyperref}
\usepackage{bm}

\usepackage{tikz}
\usetikzlibrary{positioning}	
\usepackage{tikz-3dplot}
\usepackage{ifthen}
\usepackage{msc}
\usepackage{asymptote}
\usepackage{pgfplots}
\pgfplotsset{compat=1.18}

\usepackage{subcaption} 
\usepackage{caption}

\usepackage{appendix}

\usepackage{mathrsfs}
\usepackage{CJK}

\usepackage{extarrows}

\usepackage{enumerate}
\usepackage{amsmath,amsthm,amssymb,amscd}

\usepackage{graphicx}
\usepackage{eepic,color,colordvi,amscd}

\usepackage{tikz-cd}
\usepackage{geometry}

\newtheorem{theorem}{Theorem}[section]
\newtheorem{lemma}{Lemma}[section]

\newtheorem{definition}{Definition}

\usepackage{graphics}
\usepackage{verbatim}
\makeatletter

\newcommand{\Rmnum}[1]{\expandafter\@slowromancap\romannumeral #1@}
\makeatother

\journal{Discrete \& Computational Geometry}
\begin{document}
	
	\begin{frontmatter}

		\title{A High-Dimensional Extension of Wagner's Theorem and the Geometrization of Hypergraphs \\}
		%
		%
		%
		
		\author{
			Qiming Fang$^{a}$,
			Sihong Shao$^{b}$\\[6pt]
			{\small $^{a}$ Beijing International Center for Mathematical Research, Peking University, Beijing 100871, China}\\
			{\small $^{b}$ CAPT, LMAM and School of Mathematical Sciences, Peking University, Beijing 100871, China}
		}

		\begin{abstract}
			This paper introduces a geometric representation of hypergraphs by representing hyperedges as simplices. Building on this framework, we employ homotopy groups to analyze the topological structure of hypergraphs embedded in high-dimensional Euclidean spaces. 
				Under the assumptions of the triangulation and that all $i$-th homotopy groups are trivial for $i \leq d-2$, we provide  a necessary and sufficient condition for a $d$-uniform hypergraph to be embeddable in $\mathbb{R}^d$,
				which can be regarded as a kind of high-dimensional  extension of Wagner's Theorem for planar graphs.
			Specifically, we establish that a triangulated $d$-uniform topological hypergraph embeds into $\mathbb{R}^d$ if and only if it contains neither $K_{d+3}^d$ nor $K_{3,d+1}^d$ as a minor. Here, a triangulated $d$-uniform topological hypergraph constitutes a geometrized form of a $d$-uniform hypergraph, while $K_{d+3}^d$ and $K_{3,d+1}^d$ are the high-dimensional generalizations of the complete graph $K_5$ and the complete bipartite graph $K_{3,3}$ in $\mathbb{R}^d$, respectively.
		\end{abstract}

		
		
		

	\end{frontmatter}

	
	\section{Introduction}
	
		The aim of this work is to establish a systematic correspondence between $d$-uniform hypergraphs and CW complexes, and to investigate the embedding problem of $d$-uniform hypergraphs in $\mathbb{R}^d$.  Section~\ref{1.1} presents the \emph{main result and motivation}, Section~\ref{1.2} provides a \emph{comparison with previous research}, and Section~\ref{1.3} outlines the \emph{organization of the paper}. Unlike previous studies most of which aim to identify as many forbidden minors as possible for the embedding, this work emphasizes the structures of hypergraphs, tries to keep the structure of the corresponding CW complexes as simple as possible
		and thus introduces two constraints: homotopy groups and triangulation.

	\subsection{Main result and motivation}\label{1.1}
	
	Wagner's theorem establishes that a finite graph is planar if and only if it contains neither $K_5$ nor $K_{3,3}$ as a minor, thereby characterizing the embeddability of graphs into $\mathbb{R}^2$. This raises a natural question: Does an analogous characterization exist for higher-dimensional Euclidean spaces?
	However, the complexity and non-intuitive nature of higher dimensions make it challenging to define appropriate generalizations of concepts such as planarity, complete graphs, and complete bipartite graphs. To address this, we first establish a correspondence between hyperedges and simplices, thereby endowing hypergraphs with a geometric structure.  
	For example, each edge in a graph (that is, a $2$-uniform hypergraph) is homeomorphic to a $1$-simplex. Similarly, we may regard each hyperedge of a $d$-uniform hypergraph as a $(d-1)$-simplex. 
	When studying the planar embedding problem of graphs, each edge is allowed to undergo topological deformation, provided that it does not intersect itself and that any two edges intersect only at their common vertices. Analogously, if we view each hyperedge in a $d$-uniform hypergraph as a $(d-1)$-simplex and investigate its embedding in $\mathbb{R}^d$, we should likewise allow each hyperedge to deform topologically under homeomorphism, as long as it does not self-intersect and any two hyperedges intersect only along their common $i$-faces (where $i \leq d-2$).


	Then, we utilize homotopy groups that characterize the topological structure of hypergraphs embedded in higher-dimensional spaces and define a class of hypergraphs embeddable into $\mathbb{R}^d$, termed $\mathbb{R}^d$-hypergraphs. Evidently, an $\mathbb{R}^d$-hypergraph serves as a higher-dimensional generalization of a planar graph. Correspondingly, we generalize the complete graph $K_5$ and the complete bipartite graph $K_{3,3}$ to $\mathbb{R}^d$, denoted by $K_{d+3}^d$ and $K_{3,d+1}^d$ (see Definition \ref{complete}-\ref{bipartite}), respectively. 
		Building on these definitions, we establish a necessary and sufficient condition for a $d$-uniform hypergraph to be embeddable in $\mathbb{R}^d$ under the assumptions of the triangulation and that all $i$-th homotopy groups are trivial for $i \leq d-2$. 
	Here, a triangulated $d$-uniform topological hypergraph refers to a geometrized $d$-uniform hypergraph, formal definitions are provided in Section~\ref{section5}.
	
	\begin{theorem}\label{anti-minor}
		A triangulated $d$-uniform topological hypergraph embeds into $\mathbb{R}^d$ if and only if it contains neither $K_{d+3}^d$ nor $K_{3,d+1}^d$ as a minor.
	\end{theorem}
	
	Theorem~\ref{anti-minor} characterizes the embeddability of triangulated $d$-uniform hypergraphs into $\mathbb{R}^d$, reducing the problem of determining higher-dimensional embeddability to verifying the absence of specific minors. 
	It should be noted that, due to the requirement of triangulation, Theorem~\ref{anti-minor} cannot be regarded as a direct generalization of Wagner's theorem for planar graphs. In fact, it is obvious that Theorem~\ref{anti-minor} is weaker than Wagner's theorem in $\mathbb{R}^2$.
	

		Our idea originates from the proof of the Hadwiger conjecture for $t=5$. First, Wagner’s theorem established the connection between the Hadwiger conjecture (for $t=5$) and the Four Color Theorem. Then, Heesch introduced the \textit{Discharging Method} based on planar graphs, and finally, Appel and Haken proved the Four Color Theorem using this method, thereby directly implying the case $t=5$ of the Hadwiger conjecture.  It can be seen that the key to proving the Hadwiger conjecture for $t=5$ lies in the Discharging Method. Accordingly, the central idea of this paper is to generalize the Discharging Method to $\mathbb{R}^d$, thereby developing a new tool for studying the Hadwiger conjecture in higher dimensions. We observed that the reason for a planar graph to admit the Discharging Method is that each face of the planar graph is homeomorphic to a disk. If one can find similar structures in $\mathbb{R}^d$, the Discharging Method can be extended naturally to $\mathbb{R}^d$. This constitutes the main reason why we require $i$-th homotopy groups are trivial for $i \leq d-2$. To be more specfic, we focus on the geometrization of $d$-uniform hypergraphs with an additional triangulation condition. This additional requirement also arises from the observation that, under the triangulation, the number of forbidden minors becomes significantly smaller (since triangulation destroys the structure of many complicated minors \cite{carmesin2023embedding}). In a word, our primary motivation stems from the profound connection between hypergraph embeddings in high dimensions and the Hadwiger Conjecture. This connection, which constitutes a key direction for our future work, will be briefly discussed in the final section.

	In order to yield an intuitive structural understanding of hypergraphs embedded in higher-dimensional spaces before presenting the proof of Theorem~\ref{anti-minor}, 
	we introduce key concepts including simplices, CW complexes, skeleton, homotopy, fundamental groups, and the $n$-th homotopy group $\pi_n(X)$. Standard definitions and notations not explicitly stated here follow \cite{armstrong2013basic,bondy2008graph,hatcher2002algebraic}. Note: A topological space $X$ is $n$-connected if $\pi_i(X) = 0$ for all $i \leq n$. In contrast, a graph is $k$-connected if its vertex connectivity is at least $k$. These represent fundamentally distinct concepts. To avoid ambiguity, we will explicitly specify whether $n$-connected refers to a space or a graph in all subsequent usage. To provide intuitive insight, we briefly describe the structure of $d$-uniform hypergraphs embedded in $\mathbb{R}^d$, drawing an analogy to planar graphs. Ignoring cut edges, any planar graph admits an ear decomposition; equivalently, every $2$-connected planar graph can be viewed as a cycle with attached paths. Generalizing this structure to higher dimensions: Interpret each $(d-1)$-simplex as a hyperedge of a $d$-uniform hypergraph. View an embedding of such a hypergraph into $\mathbb{R}^d$ as a union of: A single $(d-1)$-dimensional sphere, and a collection of $(d-1)$-dimensional disks, intersecting only along their boundaries.

		\subsection{Comparison with previous research}\label{1.2}
		In order to compare the above main result with existing ones, below we first review related research results concerning the embedding problems of graphs in two-dimensional closed surfaces and of polytope and simplicial complexes in higher-dimensional Euclidean spaces. 
		
		\subsubsection{Two-dimensional closed surfaces}

	In two-dimensional closed surfaces the Ringel-Youngs theorem \cite{mohar2001graphs} states that $\chi(G)\leq \frac{1}{2} (7+\sqrt{49-24c'})$ for any graph $G$ embeddable in a surface $\Sigma$ of Euler characteristic $c'$.
	In 1979, Filotti et al. \cite{filotti1979determining} devised a polynomial-time algorithm with complexity $O(n^{\alpha k+\beta})$ to determine embeddability on orientable surfaces. This algorithm was subsequently optimized, culminating in Mohar's $O(n)$-time solution \cite{mohar1996embedding,mohar1999linear} in 1996. 
	
	For three-dimensional space, Bothe introduced the concept of linkless embedding \cite{sachs2006spatial} in 1973: an embedding of an undirected graph into $\mathbb{R}^3$ where no two cycles are linked. The related concept of knotless embedding \cite{conway1983knots} is defined analogously. 
	Cohen et al. \cite{cohen1997three} proved in 1995 that every finite graph embeds into $\mathbb{R}^3$. Since graphs are $1$-dimensional CW complexes, this result implies that studying graph embeddings into $\mathbb{R}^4$ or higher dimensions is trivial. 
	
	Consequently, investigating embedding problems for graphs or hypergraphs in higher-dimensional Euclidean spaces requires first representing them geometrically as higher-dimensional CW complexes. Motivated by this perspective, we initiate our study by geometrizing hypergraphs and subsequently establishing definitions and theorems for their embeddings in high-dimensional spaces.

		\subsubsection{Higher-dimensional Euclidean spaces}
		
		In higher-dimensional Euclidean spaces, previous research can be broadly divided into two main categories:  
		(i) the embedding problem of $n$-dimensional polyhedra in $\mathbb{R}^{2n}$;
		(ii) the embedding problem of $n$-dimensional polyhedra in $\mathbb{R}^{m}$ when $n+1 \leq m \leq 2n-1$.

		For example, Wagner presented a central conjecture in the field of embeddings of simplicial complexes in \cite{wagner2012minors}: If a $k$-dimensional simplicial complex $X$ embeds topologically into $\mathbb{R}^{2k}$, denoted $X \to \mathbb{R}^{2k}$, then 
		$f_k(X) \leq C_k \cdot f_{k-1}(X)$,
		where $C_k$ is a constant that depends only on $k$. 
		

		Melikhov showed in \cite{melikhov2011combinatorics} that every dichotomial cell complex is PL homeomorphic to a sphere. There exist exactly two 3-dimensional dichotomial cell complexes, whose 1-skeleta are $K_5$ and $K_{3,3}$, and exactly six 4-dimensional ones, whose 1-skeleta --- except one --- belong to the Petersen family. 
		Sarkaria showed in \cite{sarkaria1991one} that there exists a step-by-step embedding of $K^1$ into $\mathbb{R}^2$ if $o(K^1) = 0$. 
		In \cite{weber1967plongements}, Weber proved the following result: If $2m \ge 3(n + 1)$, then an $n$-dimensional polyhedron $K$ can be embedded in $\mathbb{R}^m$ if and only if there exists an equivariant map from the deleted product $K^*$ to the sphere $S^{m-1}$.  
		As a consequence, he showed that within the same range of dimensions, an $n$-dimensional polyhedron embeds in $\mathbb{R}^m$ if and only if it quasi-embeds in $\mathbb{R}^m$.  
		Furthermore, Segal et al. showed in \cite{segal1998embeddings} that for $m \ge \max(4, n)$, the dimensional restrictions in Weber's results are indeed necessary in all cases.  
		This leaves only two open cases, namely $m = 3$ with $n = 2$ or $3$, in the corresponding embedding problems.
		%
		In \cite{carmesin2023embedding}, Carmesin mainly studied the embedding problem of $2$-complexes in $3$-space, and proved that a simply connected, locally $3$-connected, $2$-dimensional simplicial complex $C$ embeds in $3$-space if and only if $C$ has no space restriction from an explicit list $\mathcal{Z}$. The list $\mathcal{Z}$ consists of cones over subdivisions of $K_5$ and $K_{3,3}$, five similar constructions to which we refer to as combined cones, and M\"{o}bius obstructions.

		\subsubsection{The difference between our result and existing ones}
		
		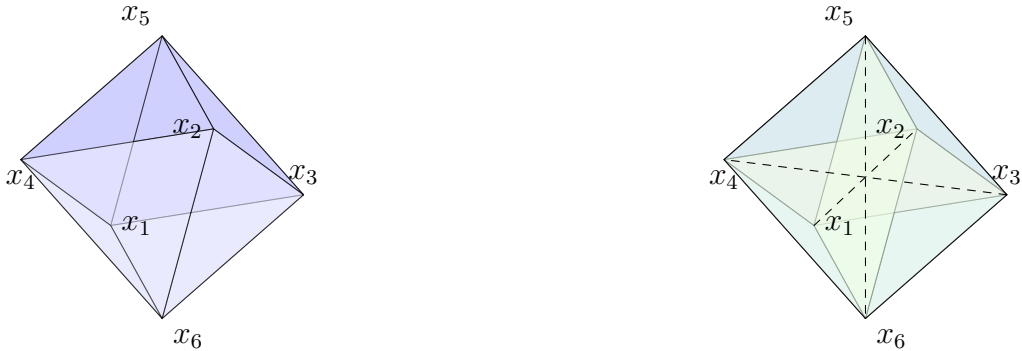
\begin{figure}[htbp]
			\centering
			\begin{subfigure}[t]{0.45\textwidth}
				\centering
				\tdplotsetmaincoords{70}{110}
				\begin{tikzpicture}[tdplot_main_coords, scale=2, line join=round, line cap=round]
					\coordinate (x1) at (1,0,0);
					\coordinate (x2) at (-1,0,0);
					\coordinate (x3) at (0,1,0);
					\coordinate (x4) at (0,-1,0);
					\coordinate (x5) at (0,0,1);
					\coordinate (x6) at (0,0,-1);
					
					\draw[fill=blue!20,opacity=0.7] (x1)--(x3)--(x5)--cycle;
					\draw[fill=blue!20,opacity=0.7] (x3)--(x2)--(x5)--cycle;
					\draw[fill=blue!20,opacity=0.7] (x2)--(x4)--(x5)--cycle;
					\draw[fill=blue!20,opacity=0.7] (x4)--(x1)--(x5)--cycle;
					
					\draw[fill=blue!10,opacity=0.6] (x1)--(x3)--(x6)--cycle;
					\draw[fill=blue!10,opacity=0.6] (x3)--(x2)--(x6)--cycle;
					\draw[fill=blue!10,opacity=0.6] (x2)--(x4)--(x6)--cycle;
					\draw[fill=blue!10,opacity=0.6] (x4)--(x1)--(x6)--cycle;
					
					\node[anchor=west] at (x1) {$x_1$};
					\node[anchor=east] at (x2) {$x_2$};
					\node[anchor=south] at (x3) {$x_3$};
					\node[anchor=north] at (x4) {$x_4$};
					\node[anchor=south east] at (x5) {$x_5$};
					\node[anchor=north west] at (x6) {$x_6$};
				\end{tikzpicture}
				\caption{Octahedron $P_1$ with faces $x_1x_3x_5$, $x_1x_4x_5$, $x_1x_3x_6$, $x_1x_4x_6$, $x_2x_3x_5$, $x_2x_4x_5$, $x_2x_3x_6$, and $x_2x_4x_6$.}
				\label{fig:p1}
			\end{subfigure}
			\hfill
			\begin{subfigure}[t]{0.45\textwidth}
				\centering
				\tdplotsetmaincoords{70}{110}
				\begin{tikzpicture}[tdplot_main_coords, scale=2, line join=round, line cap=round]
					\coordinate (x1) at (1,0,0);
					\coordinate (x2) at (-1,0,0);
					\coordinate (x3) at (0,1,0);
					\coordinate (x4) at (0,-1,0);
					\coordinate (x5) at (0,0,1);
					\coordinate (x6) at (0,0,-1);
					
					\draw[fill=blue!20,opacity=0.7] (x1)--(x3)--(x5)--cycle;
					\draw[fill=blue!20,opacity=0.7] (x3)--(x2)--(x5)--cycle;
					\draw[fill=blue!20,opacity=0.7] (x2)--(x4)--(x5)--cycle;
					\draw[fill=blue!20,opacity=0.7] (x4)--(x1)--(x5)--cycle;
					
					\draw[fill=blue!10,opacity=0.6] (x1)--(x3)--(x6)--cycle;
					\draw[fill=blue!10,opacity=0.6] (x3)--(x2)--(x6)--cycle;
					\draw[fill=blue!10,opacity=0.6] (x2)--(x4)--(x6)--cycle;
					\draw[fill=blue!10,opacity=0.6] (x4)--(x1)--(x6)--cycle;
					
					\draw[fill=red!10,opacity=0.6] (x1)--(x3)--(x2)--(x4)--cycle;
					\draw[fill=yellow!10,opacity=0.6] (x1)--(x5)--(x2)--(x6)--cycle;
					\draw[fill=green!10,opacity=0.6] (x3)--(x5)--(x4)--(x6)--cycle;
					
					\draw[dashed] (x1)--(x2);
					\draw[dashed] (x3)--(x4);
					\draw[dashed] (x5)--(x6);
					
					\node[anchor=west] at (x1) {$x_1$};
					\node[anchor=east] at (x2) {$x_2$};
					\node[anchor=south] at (x3) {$x_3$};
					\node[anchor=north] at (x4) {$x_4$};
					\node[anchor=south east] at (x5) {$x_5$};
					\node[anchor=north west] at (x6) {$x_6$};
				\end{tikzpicture}
				\caption{Polyhedron $P_2$, obtained by adding three squares ($x_1x_3x_2x_4$, $x_1x_5x_2x_6$, $x_3x_5x_4x_6$) to $P_1$.}
				\label{fig:p2}
			\end{subfigure}
			
			\caption{The key difference between this paper and previous studies lies in the treatment of $P_2$.
				Although $P_1$ and $P_2$ share the same $1$-skeleton, namely the complete tripartite graph $K_{2,2,2}$, $P_1$ can be embedded in $\mathbb{R}^3$, while $P_2$ cannot \cite{carmesin2023embedding}.
				Thus previous studies regard $P_2$ as a forbidden minor in $\mathbb{R}^3$. Below we are able to show that the triangulation of $P_2$ yields $K_6^3$: 
				First, each of the three squares $x_1x_3x_2x_4$, $x_1x_5x_2x_6$, and $x_3x_5x_4x_6$ must be subdivided into $2$-simplices.  
				Since we focus on simple hypergraphs --- where multiple edges are not exist --- the triangulation must be carried out as follows:
				(i) Add the edge $x_1x_2$ within $x_1x_3x_2x_4$, subdividing it into two $2$-simplices: $x_1x_2x_3$ and $x_1x_2x_4$;
				(ii) Add the edge $x_5x_6$ within $x_1x_5x_2x_6$, subdividing it into two $2$-simplices: $x_1x_5x_6$ and $x_2x_5x_6$;
				(iii) Add the edge $x_3x_4$ within $x_3x_5x_4x_6$, subdividing it into two $2$-simplices: $x_3x_4x_5$ and $x_3x_4x_6$.
				Next, we further add the $2$-simplices $x_1x_2x_5$, $x_1x_2x_6$, $x_3x_5x_6$, $x_4x_5x_6$, $x_1x_3x_4$, and $x_2x_3x_4$ to $P_2$.  
				Let $P_2^\Delta$ denote the resulting triangulated complex. It is straightforward to verify that every polyhedral $3$-cell in $P_2^\Delta$ is a tetrahedron.  
				Moreover, for any three vertices $\{x_i, x_j, x_k\}\subseteq \{x_1, x_2, x_3, x_4, x_5, x_6\}$, there exists a $2$-simplex $x_ix_jx_k$ in $P_2^\Delta$.  
				Therefore, $P_2^\Delta$ corresponds precisely to $K_6^3$ as described in Theorem~\ref{anti-minor}. }
			\label{fig:comparison}
		\end{figure}

		Most previous studies were conducted from the perspective of complexes or polytope, focusing primarily on their classification and structural properties and aiming to identify as many forbidden minors as possible for the embedding. To achieve this, they usually considered the worst case by adding as many faces as possible while keeping the $1$-skeleton unchanged (see e.g. \cite{carmesin2023embedding}) and hardly raised any restrictions on the homotopy groups. However, 
		this paper tris to keep the structure of the CW complexes as simple as possible and thus introduces two structural constraints --- homotopy groups (Definition~\ref{constraint}) and triangulation (Definition~\ref{triangulated-topological-hypergraph}) --- to simplify the CW complexes. We indeed show that, under the triangulation condition, all forbidden minors in $\mathbb{R}^d$ are transformed into either $K_{d+3}^d$ or $K_{3,d+1}^d$. Within this framework, we establish a necessary and sufficient condition for determining whether a triangulated $d$-uniform topological hypergraph can be embedded into $\mathbb{R}^d$. Figure~\ref{fig:comparison} gives an example in $\mathbb{R}^3$ and shows how a forbidden minor in previous studies can be transformed into $K_6^3$ through triangulation. In higher-dimensional settings, forbidden minors similarly reduce to this simple form after triangulation. Furthermore, the use of homotopy group constraints and triangulation may reduce the number of forbidden minors, which aids in solving graph coloring problems. Our subsequent research will focus on bounding the chromatic number via high dimensional embedding.

	\subsection{Organization of the paper}\label{1.3}
	
	The rest of this paper is organized as follows. Section~\ref{section3} mainly talks about the definitions of hypergraphs that can be embedded in $\mathbb{R}^d$. In Section~\ref{section5}, we give the definitions of closed $\mathbb{R}^d$-hypergraph and triangulated $\mathbb{R}^d$-hypergraph. In Sections~\ref{section-Br} and \ref{section-ED}, we prepare the groundwork for proving Theorem~\ref{anti-minor}, and the proof is laid out in Section~\ref{recognizing}. Section~\ref{FW} briefs the significance of this work and outlines some exploring directions.

	%
	%
	%
	%

	\section{$\mathbb{R}^d$-hypergraph: Definition and property}
	\label{section3}
	
	Topologically, a planar graph is well understood as a finite union of $1$-dimensional spheres (cycles) and $1$-dimensional balls (paths), forming a structure embeddable in $\mathbb{R}^2$. A natural generalization involves considering a finite union of $(d-1)$-dimensional spheres and $(d-1)$-dimensional balls. If such a structure embeds into $\mathbb{R}^d$, it constitutes a higher-dimensional analogue of a planar graph. To formalize this, we introduce algebraic topology tools for concise description. By interpreting $(d-1)$-dimensional simplices as hyperedges of a $d$-uniform hypergraph, we extend planar graphs to higher dimensions. Just as line segments ($1$-simplices) in planar graphs permit topological deformation, we allow analogous flexibility for hyperedges embedded in $\mathbb{R}^d$.

		\begin{definition}[faces of a $k$-simplex]
			Let $\sigma = [v_0, v_1, \dots, v_k]$ be a $k$-simplex, that is, the convex hull of $k+1$ affinely independent points in $\mathbb{R}^d$. 
			For any nonempty subset $I \subseteq \{0,1,\dots,k\}$, the convex hull of $\{v_i \mid i \in I\}$ is called an \emph{$r$-face} of $\sigma$, where $r = |I|-1$. 
			
			In particular,
			\begin{itemize}
				\item $0$-faces are called \emph{vertices};
				\item $1$-faces are called \emph{edges};
				\item $2$-faces are called \emph{triangular faces} (or simply \emph{faces});
				\item $(k-1)$-faces are called \emph{facets}.
			\end{itemize}
			The simplex $\sigma$ itself is regarded as its unique $k$-face.
		\end{definition}

	Building upon the aforementioned definitions, we proceed to geometrize $d$-uniform hypergraphs in order to examine their embedding properties in high-dimensional spaces. 
	
	\begin{definition}[{\it $d$-uniform-topological hypergraph}]\label{topological-hypergraph}
		Let $H = (X,E)$ be a $d$-uniform hypergraph. If each hyperedge $e$ of $H$ is regarded as a $(d-1)$-simplex, each vertex $v\in V(e)$ is regarded as a $(d-2)$-face of $e$, we refer to such a hypergraph as a {\it $d$-uniform-topological hypergraph}. 
		
	\end{definition}
	
	Since every $(d-1)$-simplex contains $d$ $(d-2)$-faces, each of which is a $(d-2)$-simplex, it follows that any $d$-uniform hypergraph can be geometrized. 
	According to the above definition, after geometrization, each hyperedge $e\in E(H)$ becomes a $(d-1)$-simplex, while each vertex $v\in V(e)$ becomes a $(d-2)$-face of $e$. 
	
	It should be noted that a vertex in a $d$-uniform hypergraph corresponds to a $(d-2)$-simplex in the associated $d$-uniform topological hypergraph.  
		In contrast, a vertex in a $d$-uniform topological hypergraph actually represents a $0$-face.  
		To avoid ambiguity, throughout the remainder of this paper, all vertices mentioned refer exclusively to $0$-faces.

		By analogy with the definition of an induced subhypergraph, we define an \emph{induced sub-$d$-uniform-topological hypergraph} as follows. 
		
		\begin{definition}[induced sub-$d$-uniform-topological hypergraph]
			Let $H = (V, E)$ be a $d$-uniform-topological hypergraph, where $V$ is the vertex set and $E \subseteq 2^V$ is the set of hyperedges. 
			For a subset of vertices $U \subseteq V$, the \emph{induced sub-$d$-uniform-topological hypergraph} of $H$ on $U$, denoted by $H[U]$, is defined as $H[U] = (U, E_U)$, where $E_U = \{\, e \in E \mid e \subseteq U \,\}$.
			In other words, $H[U]$ consists of all hyperedges of $H$ that are entirely contained in $U$. 
		\end{definition}

	\begin{definition}[homotopy group constraint]\label{constraint}
		Let $T_1, T_2, ...,T_m$ be $(d-1)$-simplices, $K= \bigcup \limits_{i = 1}^m T_i$ be a $d$-uniform-topological hypergraph. If $K$ satisfies the following conditions, we say that $K$ satisfies the \emph{homotopy group constraint}.  
		
		\begin{itemize}
			\item If $T_i$ and $T_j$ are $(d-1)$-simplices in $\{T_1, T_2, ...,T_m\}$ and $T' = T_i\cap T_j \neq \emptyset$, then $T'$ must be a face of both $T_i$ and $T_j$. 
			\item The $i$-th homotopy group of $K$ is trivial for $i\in \{1,2,3,...,d-2\}$. 
		\end{itemize}
		
	\end{definition}
	
		{\bf Remark.} Unless otherwise stated, all $d$-uniform topological hypergraphs considered in this paper are assumed to satisfy the homotopy group constraint. 
	
	In planar graphs, each face is homeomorphic to a $2$-ball.  
	Similarly, a $d$-uniform topological hypergraph that satisfies the homotopy group constraint exhibits an analogous structure, we will prove it can be regarded, up to isomorphism, as a union of internally disjoint $(d-1)$-spheres and $(d-1)$-balls in the following part. 
	Note that {\it internally disjoint} means that their interiors do not overlap and they intersect only at their boundaries. (An {\it $i$-sphere} $S^i$ is a topological space that is homeomorphic to a standard $i$-sphere. The space enclosed by an $i$-sphere is called an {\it $(i+1)$-ball} or {\it $(i+1)$-disc} $B^{i+1}$.)

	\begin{definition}[{\it $\mathbb{R}^d$-hypergraph} and {\it non-$\mathbb{R}^d$-hypergraph}] \label{special0}
		If a $d$-uniform-topological hypergraph $K$ can be embedded in $\mathbb{R}^d$, then $K$ is called an {\it $\mathbb{R}^d$-hypergraph}; otherwise, $K$ is called a {\it non-$\mathbb{R}^d$-hypergraph}. 
	\end{definition}
	
	According to Definition \ref{topological-hypergraph} and Definition \ref{special0}, we know that an $\mathbb{R}^d$-hypergraph is a special type of $d$-uniform-topological hypergraph, a $d$-uniform-topological hypergraph may not necessarily be embeddable in $\mathbb{R}^d$. 
	Note that the generalized Poincar\'{e} conjecture ensures that the $\mathbb{R}^d$-hypergraph can always be embedded in the $d$-sphere.

	

	In the process of triangulating a planar graph, it is often necessary to add some edges. However, in higher dimensions, such operations become much less intuitive. It is easy to observe that contracting simplices in an $\mathbb{R}^d$-hypergraph does not alter its homotopy groups. Nevertheless, when adding or removing simplices from an $\mathbb{R}^d$-hypergraph, one must impose certain constraints to ensure that the homotopy type remains unchanged. Therefore, we need to establish the following lemma.

	\begin{lemma} [{\it Construction of $\mathbb{R}^d$-hypergraph}] \label{special}
		Every $\mathbb{R}^d$-hypergraph can be constructed by the following procedure. 
		
		{\bf Procedure \uppercase\expandafter{\romannumeral10}:} $T_1, T_2, ...,T_m$ are $(d-1)$-simplices in $\mathbb{R}^d$. Starting from $T_0$, add $T_1, T_2, ..., T_m$ in $\mathbb{R}^d$ one by one, and this procedure satisfies the following conditions: 
		
		\begin{itemize}
			\item 1. Let $T_i$ and $T_j$ be arbitrary simplex in $\{T_1, T_2, ...,T_m\}$ and $T' = T_i\cap T_j\neq \emptyset$, then $T'$ is a face of both $T_i$ and $T_j$.
			
			\item 2. Suppose the procedure is at step $i$ ($T_i$ has been added in $\mathbb{R}^d$). 
			Let $K_i= \bigcup \limits_{j = 0}^i T_j$, 
			then $T_{i+1}$ satisfies the following condition when adding $T_{i+1}$ to $\mathbb{R}^d$: $T_{i+1} \cap K_i \cong D^{d-1}$ or $T_{i+1} \cap K_i \cong S^{d-1}$. 
			
		\end{itemize}

	\end{lemma}
	
	\begin{proof}
		
		Lemma \ref{attach}-\ref{attach2} directly implies Lemma \ref{special}.
		
	\end{proof}

	\begin{lemma}\label{attach}
		Let $X$ be a $d$-dimensional CW complex satisfying $\pi_k(X) = 0$ for all $k \leq d-1$. Let Y be a $d$-dimensional simplex and $X \cap Y \cong S^{d-1}$, i.e., it is homeomorphic to the $(d-1)$-dimensional sphere. 
		Then: $\pi_k(X \cup Y) = 0$ for all $k \leq d-1.$
		
	\end{lemma}
	
	\begin{proof}
		Since $Y \cong D^d$ and $X \cap Y \cong S^{d-1}$ is the boundary of $Y$, the pair $(X \cup Y, X)$ deformation retracts onto the pair $(Y, S^{d-1})$. Therefore,
		\[
		\pi_k(X \cup Y, X) \cong \pi_k(D^d, S^{d-1}).
		\]
		
		It is a standard fact that the relative homotopy groups satisfy
		\[
		\pi_k(D^d, S^{d-1}) = 0 \quad \text{for all } k \leq d - 1.
		\]
		
		Consider the long exact sequence \cite{switzer2017algebraic} of homotopy groups for the pair $(X \cup Y, X)$:
		\[
		\cdots \to \pi_k(X) \xrightarrow{i_*} \pi_k(X \cup Y) \xrightarrow{\partial} \pi_k(X \cup Y, X) \xrightarrow{\delta} \pi_{k-1}(X) \to \cdots
		\]
		Since $\pi_k(X \cup Y, X) = 0$ for $k \leq d - 1$, and $\pi_k(X) = \pi_{k-1}(X) = 0$ for $k \leq d - 1$ by assumption, the sequence reduces to:
		\[
		0 \to \pi_k(X \cup Y) \to 0,
		\]
		which implies:
		\[
		\pi_k(X \cup Y) = 0 \quad \text{for all } k \leq d - 1.
		\]
		
	\end{proof}
	
	Using a similar approach, we can also prove the following lemma.
	
	\begin{lemma}\label{attach2}
		Let $X$ be a $d$-dimensional CW complex satisfying $\pi_k(X) = 0$ for all $k \leq d-1$. Let Y be a $d$-dimensional simplex and $X \cap Y \cong D^{d-1}$, i.e., it is homeomorphic to the $(d-1)$-dimensional sphere. 
		Then: $\pi_k(X \cup Y) = 0$ for all $k \leq d-1.$
	\end{lemma}
	

	In high-dimensional spaces, some common definitions can be generalized as follows. 
	
	
	
	\begin{definition}[{\it multiple simplices}]\label{mul}
		Let $T_1$ and $T_2$ be $i$-simplices of a $d$-uniform-topological hypergraph $G$. If $V(T_1)= V(T_2)$, then $T_1$ and $T_2$ are called $i$-dimensional {\it multiple simplices}. 
	\end{definition}

	\begin{definition}[{\it $\mathbb{R}^d$-loops}]\label{loop}
		Let $T_1$ be an $i$-simplex of a $d$-uniform-topological hypergraph. $V(T_1)=\{u_0, u_1, ..., u_k\}$. If there exists $u_i, u_j \in V(T_1)$ ($i\neq j$) such that $u_i= u_j$, then $T_1$ is called an {\it $\mathbb{R}^d$-loop}. 
	\end{definition}
	
	Since higher-dimensional simplices have more than two vertices, the definition of the loops in higher-dimensional manifolds differs slightly from that in planar graphs. As long as two vertices of a simplices overlap, we consider it as an $\mathbb{R}^d$-loop.

	\begin{definition}[{\it simple $\mathbb{R}^d$-hypergraph}]\label{simple} 
		If a $d$-uniform-topological hypergraph (or $\mathbb{R}^d$-hypergraph) $G$ does not contain multiple simplices and $\mathbb{R}^d$-loop, then $G$ is called a {\it simple $d$-uniform-topological hypergraph (or simple $\mathbb{R}^d$-hypergraph)}. 
	\end{definition}

	Unless otherwise specified, all $d$-uniform-topological hypergraphs mentioned hereafter will be simple. Analogous to the definition of incident and adjacent in graph theory, we can define the notions of incident and adjacent in $d$-uniform-topological hypergraphs. 
	
	\begin{definition}[{\it neighbour}]
		Let $G$ be an $\mathbb{R}^d$-hypergraph, and an $i$-face of $G$ is denoted by $a_i$, and the set of $i$-face of $G$ is denoted by 
		$$A_i(G)= \{a_i| a_i \ is \ \emph{an} \ i\emph{-face of} \ G \}.$$ 
		For convenience, we use $V(G)$ to denote the vertex set of $G$, use $V(a_i)$ to denote the vertex set of $a_i$.

	\end{definition}

	\begin{definition}[{\it incident and adjacent}] \label{def-incident}
		Let $u$ and $v$ be vertices, we say $u$ is adjacent to $v$ if there exists $a_{d-1} \in A_{d-1}(G)$ such that $u, v \in V(a_{d-1})$. 
		The set of all vertices that adjacent to $u$ is denoted by $N_G(u)$, the degree of $u$ is denoted by $d_G(u)= |N_G(u)|$. 
		
		Let $a_i$ be an $i$-face, and $a_j$ be a $j$-face ($i>0$ and $i\leq j$). We say $a_i$ is incident to $a_j$ (or $a_j$ is incident to $a_i$) if $i< j$ and $a_i\cap a_j= a_i$; we say $a_i$ is adjacent to $a_j$ if $i= j$ and $a_i\cap a_j$ is an $(i-1)$-face. 
		The set of all $j$-face incident (adjacent) to $a_i$ is denoted by $N_{Gj}(a_i)$. We say $d_{Gj}(a_i)= |N_{Gj}(a_i)|$ is the {\it $j$-dimensional degree} of $a_i$.

	\end{definition}

		We now introduce the definition of \emph{minor} in the framework of $d$-uniform topological hypergraphs,
		and start with defining the operations of \emph{deletion}, \emph{merging}, as well as \emph{contraction} of simplices.
		
		\vspace{0.5em}
		
		\begin{definition}[{\it simplex deletion}]\label{deletion}
			Given a $d$-uniform topological hypergraph $G$, there are two natural ways of deriving smaller hypergraphs from $G$. 
			If $e$ is a $(d-1)$-simplex of $G$, we may obtain a hypergraph with $m-1$ $(d-1)$-simplices by deleting $e$ from $G$ but leaving the vertices and the remaining simplices intact. The resulting hypergraph is denoted by $G\backslash e$. Similarly, if $v$ is an $i$-simplex of $G$ with $i<d-1$, we may obtain a hypergraph by deleting from $G$ the simplex $v$ together with all the $(d-1)$-simplices incident with $v$. The resulting hypergraph is denoted by $G - v$ or $G\backslash v$. 
			
		\end{definition}

		Intuitively, a \emph{simplex contraction} can be understood as continuously shrinking a simplex into a single point.  
		We shall give a rigorous definition of this process by means of the concept of quotient space.
		
		
		\begin{definition}[simplex contraction]
			Let $P$ be a $d$-uniform topological hypergraph (or a regular CW-complex), and let $e \subset P$ be an $i$-simplex.  
			The \emph{contraction of $e$ in $P$} is the quotient space
			\[
			P/e := P/{\sim_e},
			\]
			where the equivalence relation $\sim_e$ is defined by
			\[
			x \sim_e y \iff x = y \text{ or } x, y \in e.
			\]
			That is, all points of $e$ are identified to a single vertex, while all other points of $P$ remain distinct.  
			The resulting space $P/e$ is called the $d$-uniform topological hypergraph obtained from $P$ by \emph{contracting the simplex $e$ to a vertex}. 
		\end{definition}


		It should be noted that a simplex contraction may produce multiple simplices (see Definition~\ref{mul}).  
		To ensure that the resulting $d$-uniform topological hypergraph remains simple, these multiple simplices must be removed.  
		In the case of graphs, multiple edges can simply be reduced by deleting the redundant ones.  
		However, in higher dimensions, two distinct cases must be considered:
		\begin{itemize}
			\item For multiple $(d-1)$-simplices, it suffices to delete one of them;
			\item For multiple $i$-simplices with $i < d-1$, according to Definition~\ref{deletion}, deleting one of them directly would also remove all $(d-1)$-simplices adjacent to it.  
			Therefore, we need to define the \emph{merging of multiple simplices}.
		\end{itemize}
		
		Intuitively, \emph{merging} means combining two multiple simplices into a single simplex.  
		A rigorous definition of this process will also be given using the notion of a quotient space.
		
		\begin{definition}[merging of two $i$-simplices]
			Let $P$ be a $d$-uniform topological hypergraph (or a regular CW-complex). 
			Let $f_1,f_2 \subset P$ be two closed $i$-simplices which are homeomorphic, and suppose their sets of $0$-simplices coincide:
			\[
			V(f_1) = V(f_2).
			\]
			Assume there exists an $(i+1)$-face $F$ of $P$ such that 
			\[
			\partial F = f_1 \cup f_2,
			\]
			and a homeomorphism 
			\[
			\varphi: f_1 \xrightarrow{\;\cong\;} f_2
			\]
			that restricts to the identity on vertices, i.e. $\varphi|_{V(f_1)} = \mathrm{id}_{V(f_1)}$.
			
			Define an equivalence relation $\sim_{\varphi}$ on $P$ by
			\[
			x \sim_{\varphi} \varphi(x) \quad \text{for all } x \in f_1.
			\]
			The \emph{merging} of $f_1$ and $f_2$ is the quotient space
			\[
			P/(f_1 \sim_{\varphi} f_2) := P/{\sim_{\varphi}},
			\]
			the resulting space $P/(f_1 \sim_{\varphi} f_2)$ is called the $d$-uniform topological hypergraph obtained from $P$ by \emph{merging $f_1$ and $f_2$}. 
		\end{definition}

		%
		%

		\begin{definition}[minor]
			Let $G$ be a finite, simple $d$-uniform topological hypergraph. 
			$H$ is called a minor of $G$ if $H$ can be formed from $G$ by deleting, merging and contracting simplices.  
			
		\end{definition}

	\begin{definition}[{\it $\mathbb{R}^d$-embedding}]
		Let $G$ be a $d$-uniform-topological hypergraph, an $\mathbb{R}^d$-embedding $G'$ of $G$ can be regarded as a hypergraph isomorphic to $G$ and is embedded in $\mathbb{R}^d$. 
	\end{definition}

	We conclude this section with a brief summary. The $\mathbb{R}^d$-hypergraph can be considered as an extension of the definition of the planar graph into higher-dimensional space or as a special type of CW complex. Whether it is an $\mathbb{R}^d$-hypergraph, or a non-$\mathbb{R}^d$-hypergraph, they are essentially special cases of CW complex. 
	
	A general $\mathbb{R}^d$-hypergraph requires that this CW complex has a trivial $i$-th homotopy group for $i\in \{1,2,3,...,d-2\}$. An $\mathbb{R}^d$-hypergraph requires that this CW complex has a trivial $i$-th homotopy group for $i\in \{1,2,3,...,d-2\}$, and can be embedded in $\mathbb{R}^d$. A non-$\mathbb{R}^d$-hypergraph requires that this CW complex has a trivial $i$-th homotopy group for $i\in \{1,2,3,...,d-2\}$ and cannot be embedded in $\mathbb{R}^d$. A thorough understanding of these definitions lays a solid foundation for subsequent proofs.

	\section{Closed $\mathbb{R}^d$-hypergraph and triangulated $\mathbb{R}^d$-hypergraph ($d\geq 3$)}\label{section5}

	A simple connected plane graph where all faces have degree three is called a plane triangulation, or simply a triangulation. This section extends the concept of triangulation to higher-dimensional spaces. It is well known that for any planar graph, edges can be added to eliminate all pendant edges. Without pendant edges, a planar graph admits an ear decomposition. As illustrated in Figure~\ref{pl}, the original graph can be reconstructed by starting from a cycle $v_1v_2v_3v_4v_5v_1$ and iteratively attaching paths $v_1v_6v_7v_8v_3$ and $v_8v_9v_{10}v_{11}v_{12}v_4$. As shown in Figure \ref{s3}, if we generalize the initial cycle in the ear decomposition of a planar graph to a $d$-dimensional sphere $S^d$, and generalize the paths to $d$-dimensional balls, the associated concepts can be extended to higher-dimensional spaces.

	Note that every planar graph admits a triangulation; that is, we can add edges such that each face of the graph becomes a triangle. Analogously, in higher dimensions, we can add simplices to an $\mathbb{R}^d$-hypergraph $G$ so that each maximal connected component in the set $\mathbb{R}^d\backslash G$ is homeomorphic to a $d$-dimensional simplex. The simplices in $\mathbb{R}^d$-hypergraphs can be divided into two categories which are similar to the {\it pendant edge} and {\it non-pendant edge} in planar graphs. Next, we will extend the concept of pendant edges to higher-dimensional spaces.

	\begin{figure}
		\centering
		\begin{tikzpicture}[scale=1, every node/.style={circle, draw, minimum size=6mm, inner sep=1pt}]
			\node (v1) at (0, 0) {$v_1$};
			\node (v2) at (2, 0) {$v_2$};
			\node (v3) at (4, 0) {$v_3$};
			\node (v4) at (6, 0) {$v_4$};
			\node (v5) at (3, -1) {$v_5$};
			\node (v6) at (0, 1.5) {$v_6$};
			\node (v7) at (0, 3) {$v_7$};
			\node (v8) at (2, 3) {$v_8$};
			\node (v9) at (3, 3) {$v_9$};
			\node (v10) at (4, 3) {$v_{10}$};
			\node (v11) at (5, 3) {$v_{11}$};
			\node (v12) at (6, 3) {$v_{12}$};
			
			\draw (v1) -- (v2);
			\draw (v2) -- (v3);
			\draw (v3) -- (v4);
			\draw (v4) -- (v5);
			\draw (v5) -- (v1);
			
			\draw (v1) -- (v6);
			\draw (v6) -- (v7);
			\draw (v7) -- (v8);
			\draw (v8) -- (v3);
			
			\draw (v8) -- (v9);
			\draw (v9) -- (v10);
			\draw (v10) -- (v11);
			\draw (v11) -- (v12);
			\draw (v12) -- (v4);
		\end{tikzpicture}
		\caption{Ear decomposition.}
		\label{pl}
	\end{figure}
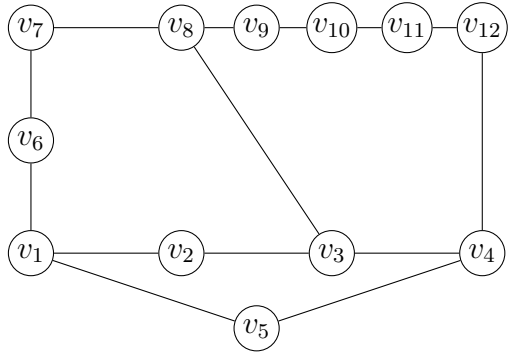

	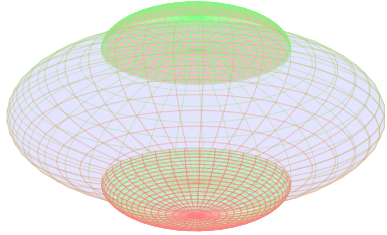
\begin{figure}
		\centering
		\begin{tikzpicture}
			\centering
			\begin{axis}[
				view={60}{30},
				hide axis,
				colormap={bluered}{color(0cm)=(red!50) color(1cm)=(green!50)},
				]
				\addplot3[
				surf,
				shader=flat,
				fill=blue!20,
				draw=none,
				opacity=0.3,
				domain=0:360,      
				y domain=0:180,    
				samples=40,
				samples y=20
				]
				({2*cos(x)*sin(y)}, {2*sin(x)*sin(y)}, {2*cos(y)});
				
				\addplot3[
				surf,
				shader=flat,
				fill=red!40,
				draw=none,
				opacity=0.5,
				domain=0:360,
				y domain=0:90,
				samples=40,
				samples y=20
				]
				({1*cos(x)*sin(y)}, {1*sin(x)*sin(y)}, {2 + 1*cos(y)});
				
				\addplot3[
				surf,
				shader=flat,
				fill=green!40,
				draw=none,
				opacity=0.5,
				domain=0:360,
				y domain=90:180,
				samples=40,
				samples y=20
				]
				({1*cos(x)*sin(y)}, {1*sin(x)*sin(y)}, {-2 + 1*cos(y)});
			\end{axis}
		\end{tikzpicture}
		\caption{An $\mathbb{R}^3$-hypergraph without pendant simplex.}
		\label{s3}
	\end{figure}



	\begin{definition}[{\it pendant simplex}]\label{pendant simplex}
		Let $T$ be a $(d-1)$-simplex of a $\mathbb{R}^d$-hypergraph $K$, $N_{K(d-2)}(T)$ represent the set of all $(d-2)$-simplex that are incident to $T$. If $\forall J \in N_{K(d-2)}(T)$, there exists a $(d-1)$-simplex $T' \subseteq K$ such that $J= T \cap T'$, then $T$ is called {\it non-pendant simplex}, if not $T$ is called {\it pendant simplex}. 
	\end{definition}
	
	Similarly, following the definitions of $2$-connected planar graphs and triangulated planar graphs, we can introduce the notions of closed $\mathbb{R}^d$-hypergraphs and triangulated $\mathbb{R}^d$-hypergraphs. 
	
	\begin{definition}[{\it closed $d$-uniform topological hypergraph}]\label{closed topological-hypergraph}
		The {\it $d$-uniform-topological hypergraph} without {\it pendant simplex} is called the {\it closed $d$-uniform topological hypergraph}. 
		
	\end{definition}

	\begin{definition}[{\it closed $\mathbb{R}^d$-hypergraph}]\label{closed}
		Let $K$ be an $\mathbb{R}^d$-hypergraph, then $K$ is called a {\it closed $\mathbb{R}^d$-hypergraph} if $K$ contains no pendant simplex. 
	\end{definition}

	\begin{definition}[{\it triangulated $d$-uniform topological hypergraph}]\label{triangulated-topological-hypergraph}
		The $(d-1)$-skeleton of a closed $\mathbb{R}^{d+1}$-hypergraph is referred to as a triangulated $d$-uniform topological hypergraph. 
	\end{definition}
	
		We now explain the meaning of the term “triangulated” in Definition~\ref{triangulated-topological-hypergraph}.  
		According to Definition~\ref{special0}, an $\mathbb{R}^{d+1}$-hypergraph $G$ consists of a collection of $d$-simplices, denoted by $\{T_1, T_2, \dots, T_m\}$. The $(d-1)$-skeleton of $G$ is therefore formed by the boundaries $\partial(T_i)$ of these simplices.  
		In other words, a triangulated $d$-uniform topological hypergraph $G'$ which is the $(d-1)$-skeleton of $G$ can be regarded as being composed of $\partial(T_i)$. 

Note that the triangulated $d$-uniform topological hypergraph is a special case of the closed $d$-uniform topological hypergraph since there is no pendant simplex in it. 

\begin{definition}[{\it triangulated $\mathbb{R}^d$-hypergraph}]\label{triangulate-graph}
	Let $K$ be a {\it triangulated $d$-uniform topological hypergraph}, then $K$ is called the {\it triangulated $\mathbb{R}^d$-hypergraph} if $K$ can be embedded in $\mathbb{R}^d$. 
	
\end{definition}

\begin{lemma}\label{structure}
	Let $K$ be a closed $\mathbb{R}^d$-hypergraph. 
	Then every connected component of the complement \(S^d\setminus K\) is homeomorphic to the open $d$-ball.
\end{lemma}

\begin{proof}
	Note that the generalized Schoenflies theorem (Lemma \ref{Schoenflies}) guarantees that no pathological cases like \emph{Alexander's horned sphere} can arise in this proof.
	
	Since $K$ can be viewed as a CW complex, the complement \(S^d\setminus K\) is partitioned into $d$-dimensional components $D_1,D_2,...,D_x$. If there exists a $D_i\in \{D_1,D_2,...,D_x\}$ such that $D_i$ is not homeomorphic to the open $d$-ball, then the boundary of $D_i$ (denoted by $\partial(D_i)$) is not $S^{d-1}$ by Lemma \ref{Schoenflies}, furthermore, $\partial(D_i)$ is not $(d-2)-connected$.  Therefore, there exists a $j$-dimensional sphere $S^j\subseteq \partial(D_i)$ ($j\leq d-2$) such that $S^j$ cannot continuously deform into a base point. On the other hand, it is clear that $S^j \subseteq K$, which implies that $K$ is not $(d-2)$-connected, a contradiction. 
	
\end{proof}

Lemma \ref{structure} implies that a closed $\mathbb{R}^d$-hypergraph can be regarded as a union of internally disjoint $(d-1)$-dimensional spheres.

\begin{lemma}[generalized Schoenflies theorem \cite{brown1961proof}]\label{Schoenflies}
	Let $\varphi\colon S^{n-1} \hookrightarrow S^n$ be a topological embedding in a locally flat way (that is, the embedding extends to that of a thickened sphere) with $n \geq 2$, and let $A$ be the closure of a component of $S^n \setminus \varphi(S^{n-1})$, then $A$ is homeomorphic to the closed $n$-dimensional ball $D^n$.
\end{lemma}

\section{Bridges}\label{section-Br}
In this section, we aim to establish some lemmas of {\it bridge} in higher-dimensional spaces. 
Let $H$ be a proper subgraph of a connected $\mathbb{R}^d$-hypergraph $G$. The set $A_{d-1}(G) \backslash A_{d-1}(H)$ may be partitioned into classes as follows. 
For each component $F$ of $G[V(G)-V(H)]$, there is a class consisting of the {\color{blue} $d$-simplices} of $F$ together with the {\color{blue} $d$-simplices} linking $F$ to $H$. 
Each remaining $d$-simplex $e$ defines a singleton class $\{e\}$. 
The subgraphs of $G$ induced by these classes are the bridges of $H$ in $G$. 
It follows immediately from this definition that bridges of $H$ can intersect only in $i$-simplices of $H$ with $i\leq d-2$, and that any two vertices of a bridge of $H$ are connected by a path in the bridge that is internally disjoint from $H$. 

For a bridge $B$ of $H$, the {\it projection} of $B$ is denoted by $p(B)= B \cap H$; 
the elements of $V(B \cap H)$ are called its {\it vertices of attachment} to $H$, the remaining vertices of $B$ are its internal vertices. 
A bridge is trivial if it has no internal vertices. 
A bridge with $k$ vertices of attachment is called a $k$-bridge. Two bridges with the same vertices of attachment and same projection are equivalent bridges.

We are concerned here with bridges of $(d-1)$-sphere, and all bridges are understood to be bridges of a given $(d-1)$-sphere $S^{d-1}$. Thus, to avoid repetition, we abbreviate `bridge of $S^{d-1}$' to `bridge' in the coming discussion. 

\begin{lemma}\label{projection-connect}
	Let $G$ be an $\mathbb{R}^d$-hypergraph, $B$ be a bridge of $(d-1)$-sphere $S'$, the projection of $B$ which is denoted by $p(B)= B \cap S^{d-1}$ is a connected $\mathbb{R}^{d-1}$-hypergraph. 
\end{lemma}

\begin{proof}
	
	By Definition \ref{special0}, we only need to prove that the $i$-th homotopy group of $p(B)$ is trivial for $i\in \{1,2,3,...,d-3\}$. 
	
	Lemma \ref{structure} implies that every closed $\mathbb{R}^d$-hypergraph can be regarded as a union of $(d-1)$-dimensional spheres $S^{d-1}$. Let $S' \cup B = S_1 \cup S_2 \cup \dots \cup S_x$, where each $S_i$ $(i \in \{1, 2, \dots, x\})$ is a $(d-1)$-dimensional sphere. Without loss of generality, we assume that $S' = S_1$. 
	
	It is easy to observe that for each $S_i \in \{ S_2, S_3, \dots, S_x \}$, the intersection $D_i = S_i \cap S'$ is either empty or an $i$-dimensional ball ($i\leq d-1$). We only need to consider the case when $i= d-1$, otherwise there exists $\{S_{x_1}, S_{x_2}, ...,S_{x_s}\}\subseteq \{ S_2, S_3, \dots, S_x \}$ such that for all $S_j\in \{S_{x_1}, S_{x_2}, ...,S_{x_s}\}$, $S_j \cap S'$ is a $(d-1)$-dimensional ball and $S_i \cap S' \subseteq \bigcup_{S_j\in \{S_{x_1}, S_{x_2}, ...,S_{x_s}\}}(S_j \cap S')$. 
	
	In case when $D_i = S_i \cap S'$ is a $(d-1)$-dimensional ball, the boundary of $D_i$, denoted $\partial(D_i)$, is a $(d-2)$-dimensional sphere, note that $p(B) = \bigcup_{i \in \{2, 3, \dots, x\}} \partial(D_i)$.
	
	Since $p(B)$ can be viewed as a union of finitely many $(d-2)$-dimensional spheres, and each pair of spheres intersects only along their boundaries (i.e., for all $i, j \in \{2, 3, \dots, x\}$, the intersection $D_i \cap D_j$ is either empty or a contractible disk), it follows from Lemma 3 and Lemma 4 that $\pi_k(p(B)) \cong \pi_k(S^{d-2})$ for all $k\leq d-3$. Therefore, the $i$-th homotopy group of $p(B)$ is trivial for $i\in \{1,2,3,...,d-3\}$. 
	
\end{proof}

The projection of a $k$-bridge $B$ with $k\geq d-1$ effects a partition of $S^{d-1}$ into $r$ disjoint segments, called the segments of $B$. Two bridges avoid each other if all the vertices of attachment of one bridge lie in a single segment of the other bridge; otherwise, they overlap.

Two bridges $B$ and $B'$ are skew if $p(B)$ contains a $(d-2)$-sphere $C(B)$ as a subgraph which effects a partition of $S^{d-1}$ into two disjoint segments $\{R_1, R_2\}$, and there are distinct vertices $u,v$ in vertices of attachment of $B'$ such that $u$ and $v$ are in different segment of $\{R_1, R_2\}$, note that there is a $uv$-path $P(u,v)$ in $S^{d-1}$ such that $P(u,v)\cap C(B) \neq \phi$ by Lemma \ref{projection-connect} and Jordan-Brouwer Separation Theorem (Lemma \ref{jordan-Brouwer}). 

We give an example of skew for $S^2$. As shown in Figure \ref{new10}, both $u$ and $v$ are in the inner region of $S^2$, the bridge induced by $\{u,u_1, u_2,...,u_5\}$ is denoted by $B_1$, the bridge induced by $\{v,v_1, v_2\}$ is denoted by $B_2$. It is obvious that $B_1$ effects a partition of $S^2$ into two disjoint segments (two hemispheres), and $v_1$ and $v_2$ lie in different segments. 

\begin{figure}
	\centering     
	\includegraphics[width=0.65\linewidth]{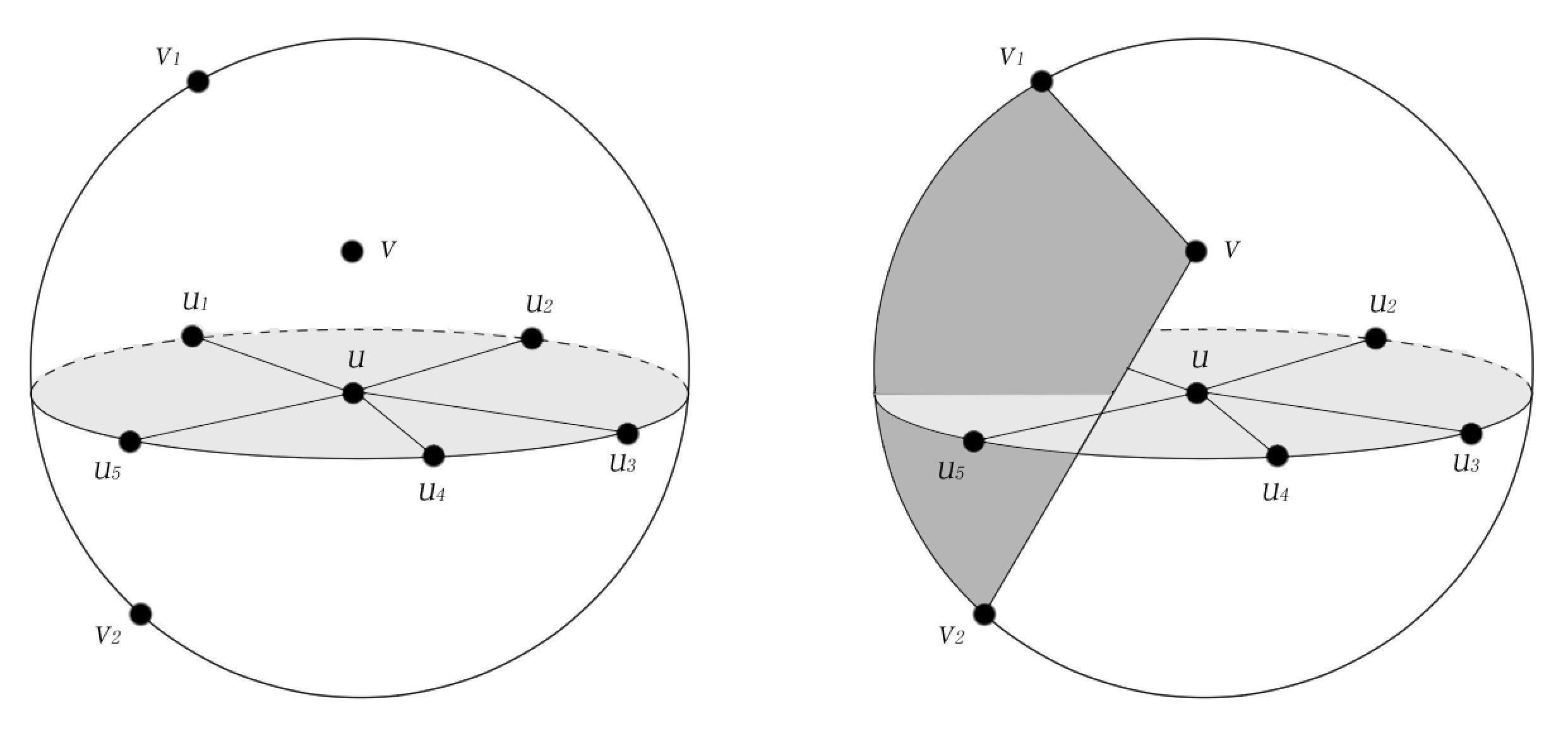}
	\caption{$B_1$ and $B_2$ are skew of $S^2$.}
	\label{new10}
\end{figure}

\begin{lemma}\label{bridges}
	Overlapping bridges of a closed $\mathbb{R}^d$-hypergraph are either skew or else equivalent $(d+1)$-bridges.
\end{lemma}

\begin{proof}
	Suppose that bridges $B$ and $B'$ overlap. Clearly, each of them must have at least $d$ vertices of attachment. If either $B$ or $B'$ is a $d$-bridge, it is easily verified that they must be skew (two equivalent $d$-bridges cannot overlap). We may therefore assume that both $B$ and $B'$ have at least $(d+1)$ vertices of attachment. 
	
	If $B$ and $B'$ are not equivalent bridges, then all the vertices of attachment of one bridge cannot lie in a single segment of the other bridge. Without loss of generality, let $C(B)\subseteq p(B)$ be a $(d-2)$-sphere which effects a partition of $S^{d-1}$ into $2$ disjoint segments $\{R_1, R_2\}$, then there exist distinct vertices $u',v'$ in vertices of attachment of $B'$ such that $u'$ and $v'$ are in different segment of $\{R_1, R_2\}$.  It follows that $B$ and $B'$ are skew. 
	
	If $B$ and $B'$ are equivalent $k$-bridges, then $k\geq d+1$. If $k\geq d+2$, $B$ and $B'$ are skew by Lemma \ref{B1}; if $k=d+1$, they are equivalent $(d+1)$-bridges.

\end{proof}

\begin{lemma}\label{B1}
	Let $G$ be an $\mathbb{R}^d$-hypergraph which is homeomorphic to $S^{d-1}$ with $|V(G)|\geq d+2$, then there is a subgraph $C$ of the $(d-2)$-skeleton of $G$ which is homeomorphic to $(d-2)$-sphere that effects a partition of $S^{d-1}$ into $2$ disjoint segments $\{R_1, R_2\}$, and there are distinct vertices $u',v'$ in vertices of attachment of $G$ such that $u'$ and $v'$ are in different segment of $\{R_1, R_2\}$. 
\end{lemma}

\begin{proof}
	Firstly, we prove that there exists a vertex $v'\in V(G)$ such that $d_{G0}(v')\leq |V(G)|-2$, if not, we know that $d_{G0}(v)=|V(G)|-1$ for every vertex $v\in V(G)$, which implies that $G$ is a complete $d$-uniform-topological hypergraph (see Definition \ref{complete}). It is impossible since $G$ is homeomorphic to $S^{d-1}$. 
	
	Let $N_{G0}(v')$ be the neighbors of $v'$, then $G[N_{G0}(v')]$ is homeomorphic to a $(d-2)$-sphere by Lemma \ref{sphere} and \ref{sphere2}, note that $G[N_{G0}(v')]$ effects a partition of $S^{d-1}$ into $2$ disjoint segments $\{R_1, R_2\}$, without loss of generality, let $v'\in R_1$. 
	
	It is easy to verify that $V(G)\backslash [N_{G0}(v') \cup \{v'\}] \neq \phi$ since $|V(G)|\geq d+2$, let $u'\in V(G)\backslash [N_{G0}(v') \cup \{v'\}]$, it follows that $u' \in R_2$. 
	
\end{proof}

\begin{lemma}\label{sphere2}
	In a closed $\mathbb{R}^d$-hypergraph $G$, if $G$ is homeomorphic to $S^{d-1}$, the $1$-skeleton of $G$ is $d$-connected. 
\end{lemma}

\begin{proof}
	Suppose that the 1-skeleton of $G$ (denoted by $G_0$) is not $d$-connected. Then there exists a $(d-1)$-cut $\{v_1, v_2, \ldots, v_{d-1}\}$ such that removing these vertices from $G_0$ disconnects the graph. Let $X$ and $Y$ be two connected components of the resulting graph. 
	
	Contract $X$ into a single vertex $x$, and contract $Y$ into a single vertex $y$. Denote the resulting graph after contraction by $G'$. Since $G'$ becomes disconnected after removing the set $\{v_1, v_2, \ldots, v_{d-1}\}$, its connectivity is strictly less than $d$.
	
	On the other hand, because $G_0$ is composed of $(d-1)$-simplices and is homeomorphic to $S^{d-1}$, the 1-skeleton of $G'$ must be the complete graph $K_{d+1}$. It follows that the connectivity of $G'$ is $d$, leading to a contradiction.
	
\end{proof}

\begin{lemma}\label{sphere3}
	In a closed $\mathbb{R}^d$-hypergraph $G$, the $1$-skeleton of $G$ is $d$-connected. 
\end{lemma}

\begin{proof}
	Let $u$ and $v$ be any two vertices in $G$. Let $S^{d-1}$ be a ${d-1}$-dimensional sphere containing both $u$ and $v$. By Lemma \ref{sphere2}, the $1$-skeleton of $S^{d-1}$ must contain $d$ internally disjoint paths connecting $u$ and $v$. 
	
	Therefore, the $1$-skeleton of $G$ must be $d$-connected. 
	
\end{proof}

\section{Hyper ear decomposition}
\label{section-ED}

In this section, we aim to establish some lemmas of {\it ear decomposition} in higher-dimensional spaces.
Let $K$ be an $\mathbb{R}^d$-hypergraph whose $1$-skeleton is $d$-connected. Note that the $d$-connected closed $\mathbb{R}^d$-hypergraph contains a subgraph $G_0$ which is homeomorphic to $S^{d-1}$. We describe here a simple recursive procedure for generating any such hypergraph starting with an arbitrary $(d-1)$-sphere of the $\mathbb{R}^d$-hypergraph. 

\begin{definition}[{\it hyper ear}]\label{ear}
	Let $F$ be a subgraph of an $\mathbb{R}^d$-hypergraph $G$. A {\it hyper ear} of $F$ in $G$ is a nontrivial $(d-1)$-ball in $G$ whose boundary lies in $F$ but whose internal vertices do not. 
\end{definition}

\begin{definition}[{\it hyper ear decomposition}]\label{hyper-ear-decomposition}
	A nested sequence of a closed $\mathbb{R}^d$-hypergraph is a sequence $(G_0,G_1,\ldots,G_k)$ of $\mathbb{R}^d$-hypergraphs such that $G_i\subseteq G_{i+1}$, $0\leq i\leq k$. A {\it hyper ear decomposition} of a $d$-connected closed $\mathbb{R}^d$-hypergraph $G$ is a nested sequence $(G_0,G_1,\ldots,G_k)$ of $d$-connected subgraphs of $G$ such that: 
	
	\begin{itemize}
		\item $G_0$ is homeomorphic to $S^{d-1}$.
		\item $G_{i+1}= G_i \cup P_i$ where $P_i$ is a hyper ear of $G_i$ in $G$ for $0\leq i\leq k$. 
		\item $G_k=G$.
	\end{itemize}
\end{definition}

\begin{lemma}\label{ear-exist}
	The closed $\mathbb{R}^d$-hypergraph $G$ with $|V(G)|\geq d+1$ has a hyper ear decomposition. 
	
\end{lemma}

\begin{proof}
	On the one hand, the $1$-skeleton of $G$ is $d$-connected by Lemma \ref{sphere3}. On the other hand, since $G$ contains at least $(d+1)$ vertices, it must contain at least one $(d-1)$-dimensional sphere $S^{d-1}$.
	
	Since $G$ is homeomorphic to the union of a finite collection of $(d-1)$-spheres by Definition \ref{special0} and Lemma \ref{structure}, it is obvious that $G$ has a hyper ear decomposition. 
\end{proof}

\begin{lemma}\label{E1}
	In a closed $\mathbb{R}^d$-hypergraph $G$ with $|V(G)|\geq d+1$, if the $1$-skeleton of $G$ is $d$-connected, then each maximal connected region or $\mathbb{R}^d\backslash G$ is bounded by a $(d-1)$-sphere. 
	
\end{lemma}

\begin{proof}
	Note that $G$ has a hyper ear decomposition by Lemma \ref{ear-exist}. 
	Consider an ear decomposition $(G_0,G_1,...,G_k)$ of $G$, where $G_0$ is homeomorphic to $S^{d-1}$, $G_k= G$, and, for $0\leq i \leq k-2$, $G_{i+1}= G_i \cup P_i$ is a $d$-connected subgraph of $G$, where $P_i$ is an ear of $G_i$ in $G$. Since $G_0$ is homeomorphic to $S^{d-1}$, the two maximal connected regions of $G_0$ are clearly bounded by a $(d-1)$-sphere. Assume, inductively, that all maximal connected regions of $G_i$ are bounded by $(d-1)$-spheres, where $i\geq 0$. Because $G_{i+1}$ is a $d$-connected $\mathbb{R}^d$-hypergraph, the ear $P_i$ of $G_i$ is contained in some maximal connected region $f$ of $G_i$. Each region of $G_i$ other than $f$ is a region of $G_{i+1}$ as well, and so, by the  induction hypothesis, is bounded by a $(d-1)$-sphere. On the other hand, the region $f$ of $G_i$ is divided by $P_i$ into two regions of $G_{i+1}$, and it is easy to see that these regions are also bounded by $(d-1)$-spheres.

\end{proof}

\begin{lemma}\label{sphere}
	In a closed $\mathbb{R}^d$-hypergraph $G$, if the $1$-skeleton of $G$ is $(d+1)$-connected, the neighbors of any vertex lie on a common $(d-1)$-sphere. 
\end{lemma}

\begin{proof}
	Let $v$ be a vertex of $G$, then the $1$-skeleton of $G-v$ is $d$-connected, so each maximal connected region of $G-v$ is bounded by a sphere by Lemma \ref{E1}. If $f$ is the region of $G-v$ in which the vertex $v$ was situated, the neighbors of $v$ lie on its bounding sphere $\partial(f)$. 
	
\end{proof}

\section{Recognizing $\mathbb{R}^d$-hypergraph}\label{recognizing}

We extend the definition of $S$-component for graphs to higher-dimensional spaces before proving the main result. 

\subsection{$S$-component}

\begin{definition}[{\it $S$-component}]
	Let $G$ be a connected $d$-uniform topological hypergraph which is not complete, let $S$ be a vertex cut of $G$, and let $X$ be the vertex set of a component of $G-S$. The subgraph $H$ of $G$ induced by $S\cup X$ is called an $S$-component of $G$. In the case where $G$ is a closed $d$-uniform topological hypergraph, the $1$-skeleton of $G$ is $d$-connected, and $S := \{x_1,x_2, ..., x_d\}$ is a $d$-vertex cut of $G$, we find it convenient to modify each $S$-component by adding a new $(d-1)$-simplex with vertex set $\{x_1,x_2, ..., x_d\}$. We refer to this simplex as a marker simplex and the modified $S$-components as marked $S$-components. The set of marked $S$-components constitutes the marked $S$-decomposition of $G$. $G$ can be recovered from its marked $S$-decomposition by taking the union of its marked $S$-components and deleting the marker hyperedge.

\end{definition}

As shown in Figure \ref{new12}, $S := \{x_1,x_2, x_3\}$ be a 3-cut of an $\mathbb{R}^3$-hypergraph $G$, 
we provide an example of the $S$-decomposition and marked $S$-decomposition of an $\mathbb{R}^3$-hypergraph. The only difference between $S$-decomposition and marked $S$-decomposition is that marked $S$-decomposition must contain a simplex with vertex set $S := \{x_1,x_2, x_3\}$. If this simplex does not exist in the original hypergraph, it needs to be added during the construction. 

\begin{figure}
	\centering     
	\includegraphics[width=1\linewidth]{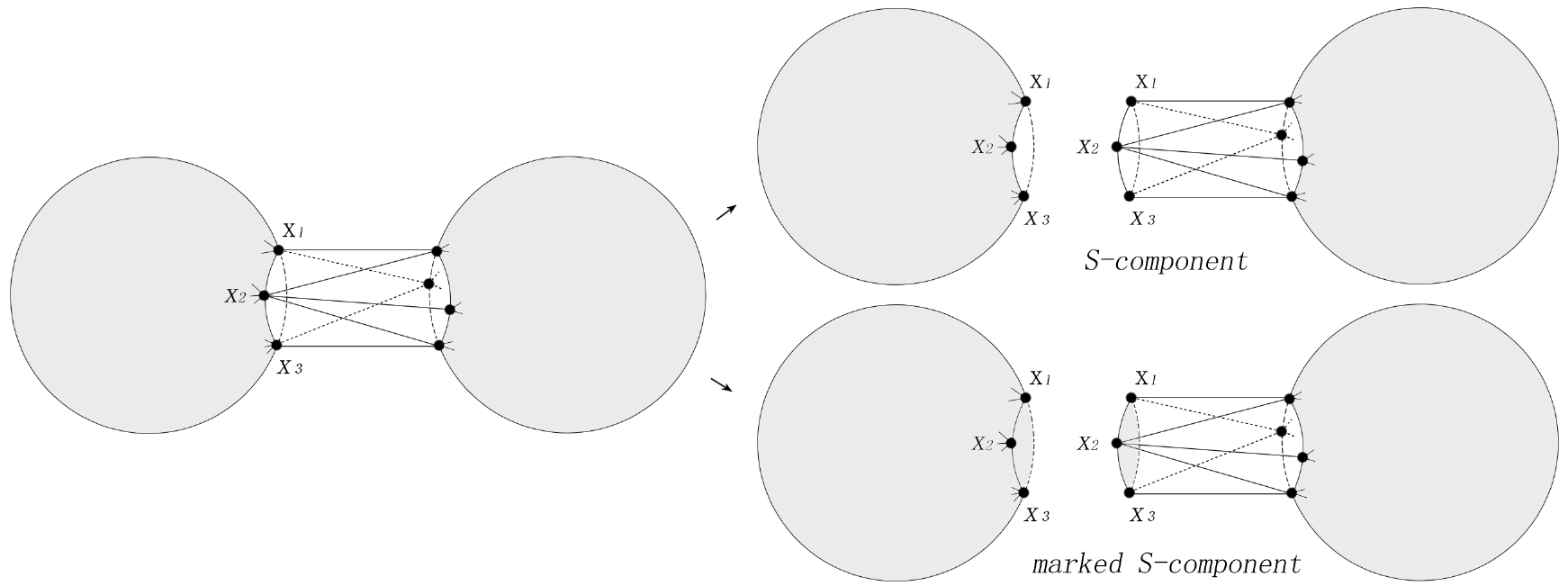}
	\caption{An example of $S$-decomposition and marked $S$-decomposition of $\mathbb{R}^3$-hypergraph.}
	\label{new12}
\end{figure}

We need to establish some lemmas before proving the main results.

\begin{lemma}\label{10.33}
	Let $G$ be a closed $\mathbb{R}^d$-hypergraph with a $d$-vertex cut $\{x_1,x_2, ..., x_d\}$, then each marked $\{x_1,x_2, ..., x_d\}$-component of $G$ is isomorphic to a minor of $G$. 
\end{lemma}

\begin{proof}
	Let $H$ be an $\{x_1,x_2, ..., x_d\}$-component of $G$, with marker simplex $e$. Let $H'$ be another $\{x_1,x_2, ..., x_d\}$-component of $G$, with marker simplex $e$, then there is a $(d-1)$-ball $B$ such that $B\subseteq H'$ and $B\cup e$ is a $(d-1)$-sphere by Lemma \ref{ear-exist}. It is easy to verify that $H$ is isomorphic to a minor of $G$ by contract $H'$ into a single simplex $e$. 
	
\end{proof}

\begin{lemma}\label{Exercise 10.4.1}
	Let $G_1$ and $G_2$ be closed $\mathbb{R}^d$-hypergraphs whose intersection is isomorphic to an $\mathbb{R}^{d-1}$-hypergraph $K_d$ with $V(K_d)= \{x_1,x_2, ..., x_d\}$ and 
	\[
	E(K_d)= \{a_i|a_i \text{ is a }(d-2)\text{-simplex with vertex set }V(K_d)\backslash \{x_i \} \text{,  }(i\in \{1,2,...,d\}) \},
	\] 
	then $G_1\cup G_2$ is a closed $\mathbb{R}^d$-hypergraph. 
\end{lemma}

\begin{proof}
	Let $H$ be a hyperplane, $V(K_d)\subseteq H$. At this point, the hyperplane $H$ divides $\mathbb{R}^d$ into two disconnected regions, denoted by $R_1$ and $R_2$, respectively. We embed $G_1$ into $R_1$ and $G_2$ into $R_2$ in such a way that $G_1$ and $G_2$ intersect only at $V(K_d)$. 
	
	By contradiction, suppose the $i$-th homotopy group of $G_1 \cup G_2$ is nontrivial for some $i \in \{1, 2, \ldots, d-2\}$, then there must exist an $i$-sphere $S^i$ that cannot be continuously contracted to the base point. If $S^i$ belongs to either $G_1$ or $G_2$, then it can be continuously contracted to the base point, which leads to a contradiction. Therefore, $S^i$ must intersect both $G_1$ and $G_2$. Let $S^i\cap G_1= L_1$ and $S^i\cap G_2= L_2$, respectively. We first transform $L_1$ into $L_3$ by homotopy, such that $L_3$ belongs to $H$. It is easy to verify that $L_2$ and $L_3$ belong to $G_2$, thus they can be continuously contracted to the base point. By combining the two homotopy transformations, we obtain that $S^i$ can be continuously contracted to the base point, a contradiction. In conclusion, the assumption is invalid, and the theorem is proven.
	
\end{proof}

\begin{lemma}\label{10.34}
	Let $G$ be an $\mathbb{R}^d$-hypergraph with a $d$-vertex cut $\{x_1,x_2, ..., x_d\}$, then $G$ is a closed $\mathbb{R}^d$-hypergraph if and only if each of its marked $\{x_1,x_2, ..., x_d\}$-components is a closed $\mathbb{R}^d$-hypergraph.
\end{lemma}

\begin{proof}
	Suppose, first, that $G$ is a closed $\mathbb{R}^d$-hypergraph. By Lemma \ref{10.33}, each marked $\{x_1,x_2, ..., x_d\}$-component of $G$ is isomorphic to a minor of $G$, hence is closed $\mathbb{R}^d$-hypergraph. 
	
	Conversely, suppose that $G$ has $k$ marked $\{x_1,x_2, ..., x_d\}$-components each of which is a closed $\mathbb{R}^d$-hypergraph. Let $e$ denote their common marker simplex. Applying Lemma \ref{Exercise 10.4.1} and induction on $k$, it follows that $G + e$ is a closed $\mathbb{R}^d$-hypergraph, hence so is $G$. 
\end{proof}

By Lemma \ref{10.34}, we know that to prove a closed $\mathbb{R}^d$-hypergraph can be embedded in $\mathbb{R}^d$, it is sufficient to show that all of its marked $\{x_1,x_2, ..., x_d\}$-components can be embedded in $\mathbb{R}^d$.

\subsection{Connectivity}

Before proving Theorem \ref{anti-minor}, we need a lemma regarding connectivity. 
\begin{lemma}
	\label{connected}
	Let $G$ be a $(d+1)$-connected graph on at least $(d+2)$ vertices, then $G$ contains an edge $e$ such that $G/e$ is $(d+1)$-connected. 
\end{lemma}

\begin{proof}
	Suppose the theorem is false. Then, for any edge $e = xy$ of $G$, the contraction $G/e$ is not $(d+1)$-connected. Let $w$ be the vertex resulting from the contraction of $e$. By Lemma \ref{not-c}, there exists vertex set $\{z_1, z_2, ..., z_{d-1}, w\}$ such that $\{z_1, z_2, ..., z_{d-1}, w\}$ is a $d$-vertex cut of $G$. 
	
	Choose the edge $e$ and the vertex set $\{z_1, z_2, ..., z_{d-1}, w\}$ in such a way that $G - \{x,y,z_1, z_2, ..., z_{d-1}\}$ has a component $F$ with as many vertices as possible. Consider the graph $G - \{z_1\}$. Because $G$ is $(d+1)$-connected, $G- \{z_1\}$ is $d$-connected. Moreover $G-\{z_1\}$ has the $d$-vertex cut $\{x,y, z_2, ..., z_{d-1}\}$. It follows that the $\{x,y, z_2, ..., z_{d-1}\}$-component $H = G[V (F) \cup \{x,y, z_2, ..., z_{d-1}\}]$ is $d$-connected. 
	
	Let $u$ be a neighbour of $z_1$ in a component of $G - \{x,y,z_1, z_2, ..., z_{d-1}\}$ different from $F$. Since $f = z_1u$ is an edge of $G$, and $G$ is a counterexample to Lemma \ref{connected}, there is a vertex set $\{v_1, v_2, ..., v_{d-1}\}$ such that $\{z_1,u, v_1, v_2, ..., v_{d-1}\}$ is a $(d+1)$-vertex cut of $G$, too. (The vertices $\{v_1, v_2, ..., v_{d-1}\}$ might or might not lie in $H$.) Moreover, because H is $d$-connected, $H-\{v_1, v_2, ..., v_{d-1}\}$ is connected (where, if there exists $v_i\in \{v_1, v_2, ..., v_{d-1}\}$ such that $v_i \in V (H)$, we set $H- v_i = H$), and thus is contained in a component of $G - \{z,u, v_1, v_2, ..., v_{d-1}\}$. But this component has more vertices than $F$ (because $H$ has $d$ more vertices than $F$), contradicting the choice of the edge $e$ and the vertex $v$.
	
\end{proof}

\begin{lemma}\label{not-c}
	Let $G$ be a $(d+1)$-connected graph on at least $(d+2)$ vertices, and let $e = xy$ be an edge of $G$ such that $G/e$ is not $(d+1)$-connected. Then there exist some vertices such that $\{x,y, z_1, z_2, ..., z_{d-1}\}$ is a $(d+1)$-vertex cut of $G$. 
\end{lemma}

\begin{proof}
	Let $\{z_1, z_2, ..., z_{d-1}, w\}$ be a $d$-vertex cut of $G/e$. At least $d-1$ of these $d$ vertices, say $\{z_1, z_2, ..., z_{d-1}\}$, is not the vertex resulting from the contraction of $e$. Set $F = G- \{z_1, z_2, ..., z_{d-1}\}$. Because $G$ is $(d+1)$-connected, $F$ is certainly $2$-connected. However $F/e = (G - \{z_1, z_2, ..., z_{d-1}\}) /e = (G/e) - \{z_1, z_2, ..., z_{d-1}\}$ has a cut vertex, namely $w$. 
	
	If $w$ is not the vertex resulting from the contraction of $e$, then $\{z_1, z_2, ..., z_{d-1}, w\}$ must be a $d$-vertex cut of $G$, a contradiction. Hence $w$ must be the vertex resulting from the contraction of $e$. Therefore $G - \{x,y,z_1, z_2, ..., z_{d-1}\} = (G/e) - \{z_1, z_2, ..., z_{d-1},w\}$ is disconnected, in other words, $\{x,y,z_1, z_2, ..., z_{d-1}\}$ is a $(d+1)$-vertex cut of $G$.
\end{proof}

\subsection{Anti-$d$-dimension minor}\label{section9}

The following proof demonstrates that there is a conclusion that holds in $\mathbb{R}^d$ that similar to Wagner's theorem. 
It is necessary to generalize the concepts of complete graphs and complete bipartite graphs to higher-dimensional spaces.

\begin{definition}[{\it complete $i$-uniform-topological hypergraph $K_n^i$}]\label{complete}
	Let $G$ be an $i$-uniform-topological hypergraph, $V$ be the vertex set of $G$ with order $n$, $\mathscr{V}(i)$ be the collection of all subsets of $V$ containing $i$ elements. 
	If for any $V_j\in \mathscr{V}(i)$, there exists a simplices $T$ in $G$ such that $V(T)= V_j$, then we call $G$ a {\it complete $i$-uniform-topological hypergraph}, which is denoted by $K_n^i$. (Figure \ref{new5} is an example of $K_4^3$. )
	
\end{definition}

\begin{figure}[ht]
	\centering
	\begin{minipage}{0.45\linewidth}
		\centering
		\includegraphics[width=0.4\linewidth]{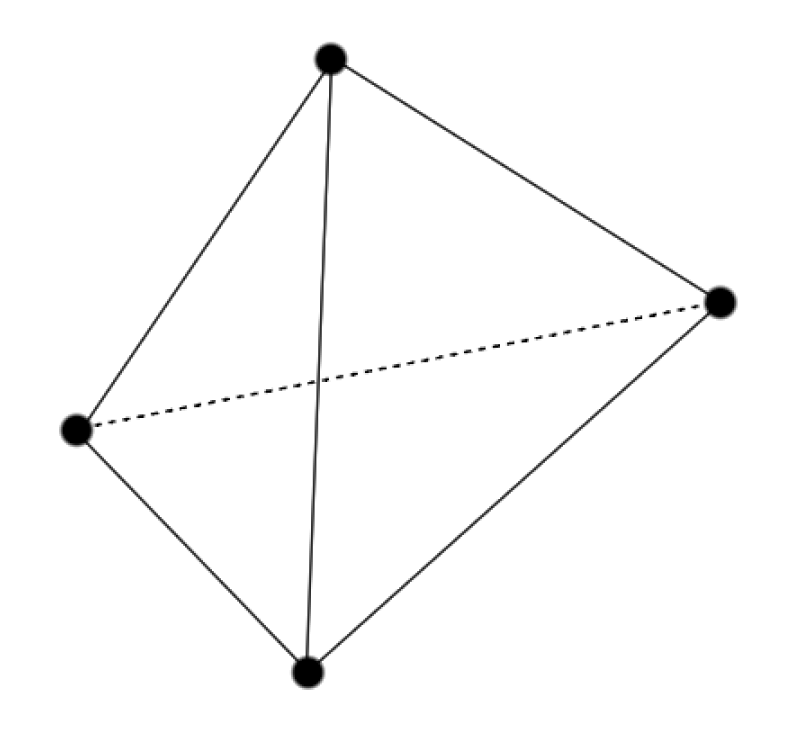}
		\caption{$K_4^3$.}
		\label{new5}
	\end{minipage}
	\hfill
	\begin{minipage}{0.45\linewidth}
		\centering
		\includegraphics[width=0.4\linewidth]{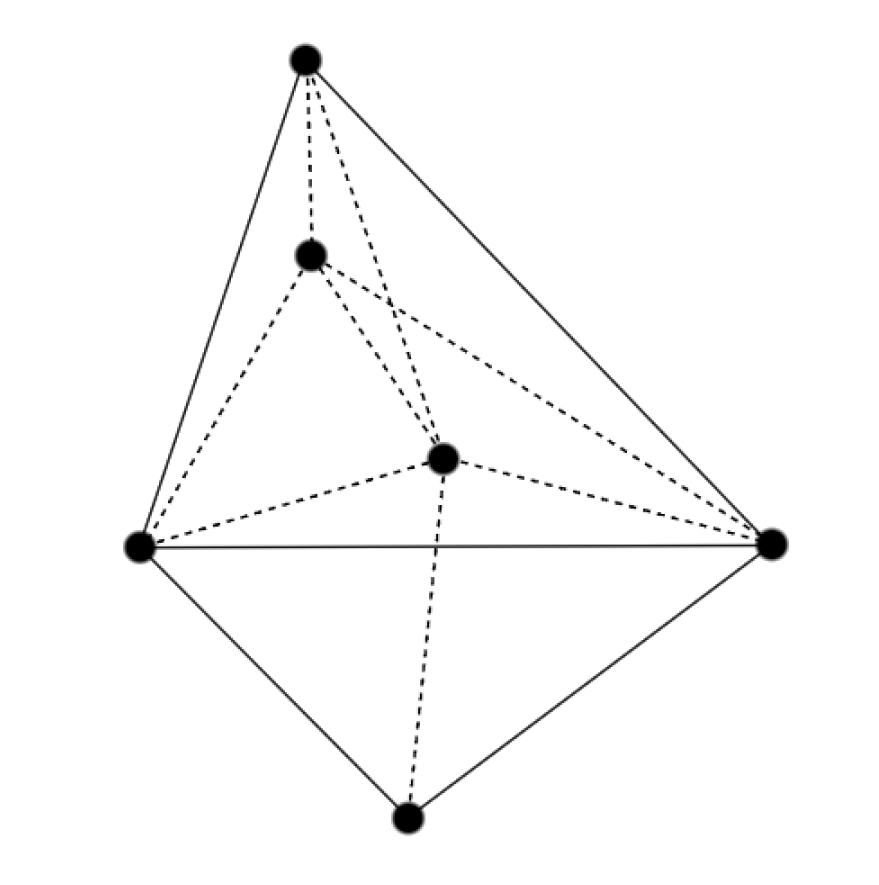}
		\caption{$K_{2,4}^3$.}
		\label{new6}
	\end{minipage}
\end{figure}



\begin{definition}[{\it complete bipartite $i$-uniform-topological hypergraph $K_{p,q}^i$}]\label{bipartite}
	Let $G$ be an $i$-uniform-topological hypergraph, $V$ be the vertex set of $G$,  $V(A:B)$ be a partition of $V$ in which $A=\{a_1, a_2, ..., a_p\}$ and $B=\{b_1, b_2, ..., b_q\}$. If $G$ satisfies the following properties, we call $G$ a {\it complete bipartite $i$-uniform-topological hypergraph}, which is denoted by $K_{p,q}^i$. (An example of $K_{2,4}^3$ is shown in Figure~\ref{new6}.)
	
	\begin{itemize}
		\item $G[B]$ is a complete $(i-1)$-uniform-topological hypergraph $K_q^{i-1}$. 
		
		\item For any $T_j\in A_{i-2}(G[B])$ and $a_k\in A$, there exists a simplex $T$ in $G$ such that $V(T)= V(T_j)\cup \{a_k\}$. 
	\end{itemize}
	
\end{definition}

Next, we use the Jordan-Brouwer Separation Theorem to prove Lemma \ref{comp} and \ref{bipa}. 

\begin{lemma}[Jordan-Brouwer Separation Theorem \cite{hatcher2002algebraic}]\label{jordan-Brouwer}
	Let $X$ be a $d$-dimensional topological sphere in the $(d+1)$-dimensional Euclidean space $\mathbb{R}^{d+1}$ ($d > 0$), i.e. the image of an injective continuous mapping of the $d$-sphere $S^d$ into $\mathbb{R}^{d+1}$, then the complement $Y$ of $X$ in $\mathbb{R}^{d+1}$ consists of exactly two connected components. One of these components is bounded (the interior) and the other is unbounded (the exterior). The set $X$ is their common boundary.
\end{lemma}

\begin{lemma}\label{comp}
	$K_{d+3}^d$ is a non-$\mathbb{R}^d$-hypergraph. 
\end{lemma}

\begin{proof}
	Suppose to the contrary that $K_{d+3}^d$ is an $\mathbb{R}^d$-hypergraph. 
	
	Let $V(K_{d+3}^d)= \{v_1, v_2, ... v_{d+2}, v_{d+3}\}$. Let $\mathscr{V}_{d+2}^{d+1}$ be the collection of all subsets of $V(K_{d+3}^d)\backslash \{v_{d+3}\}$ containing $(d+1)$ elements. 
	
	Note that for any $V_i\in \mathscr{V}_{d+2}^{d+1}$, $K_{d+3}^d[V_i]$ is homeomorphic to $S^{d-1}$.

	\begin{figure}[ht]
		\centering
		\begin{minipage}{0.45\linewidth}
			\centering     
			\includegraphics[width=1\linewidth]{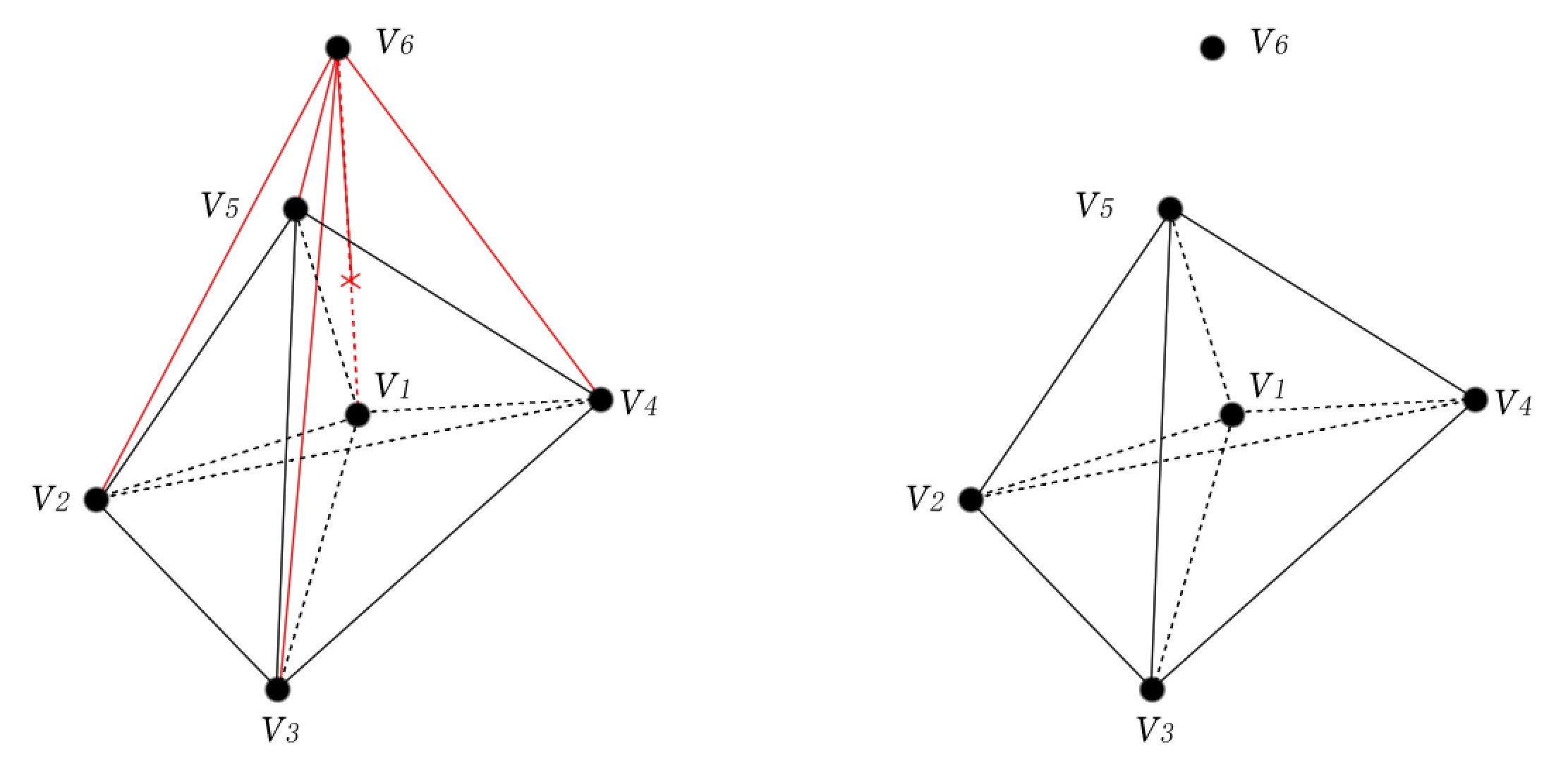}
			\caption{$K_{3,4}^3$ in $\mathbb{R}^d$.}
			\label{new7}
		\end{minipage}
		\hfill
		\begin{minipage}{0.45\linewidth}
			\centering    
			\includegraphics[width=1\linewidth]{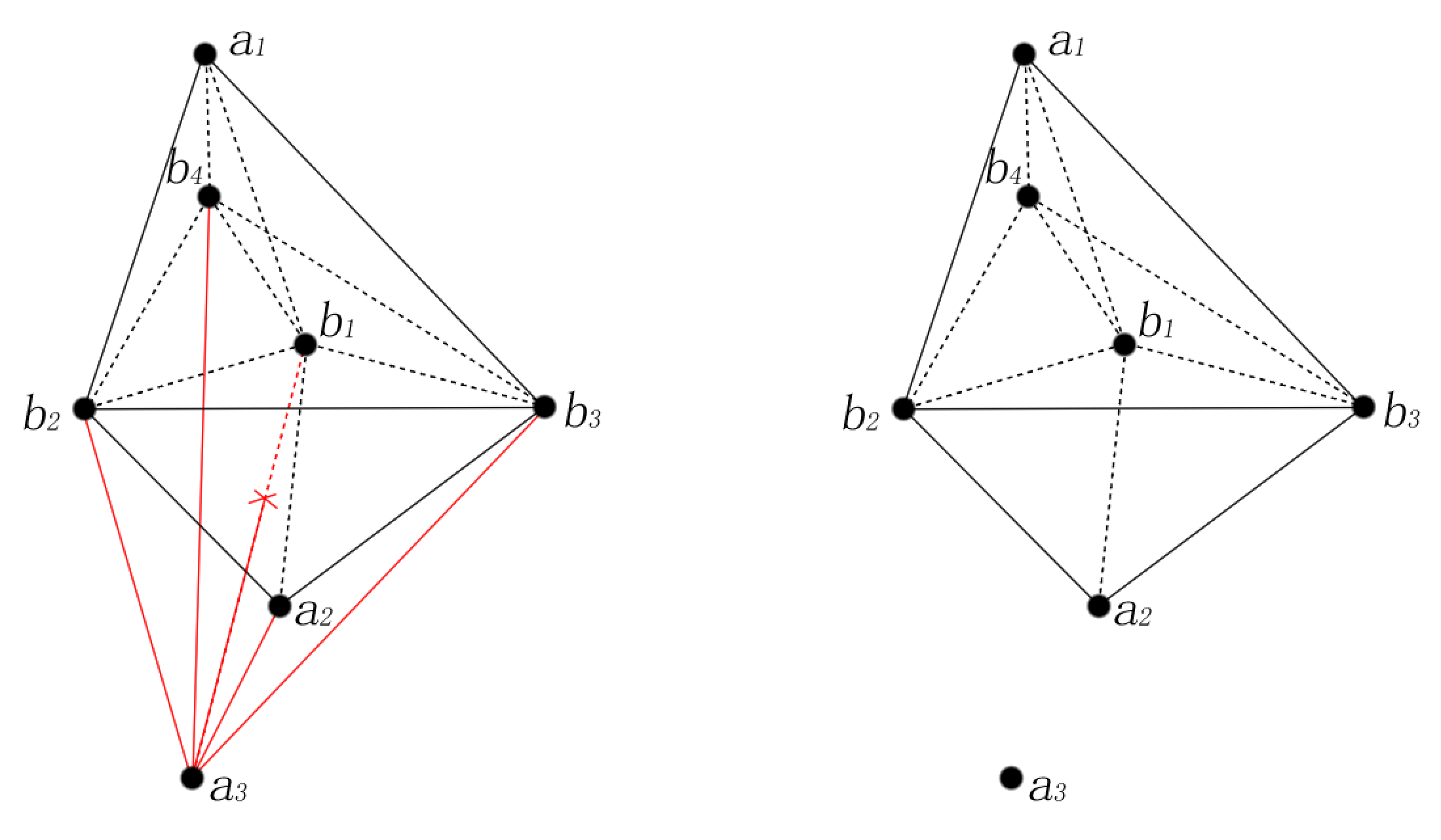}
			\caption{$K_{3,d+1}^d$ in $\mathbb{R}^d$.}
			\label{new8}
		\end{minipage}
	\end{figure}

	
	Let $V_i= V(K_{d+3}^d)\backslash \{v_{d+3}, v_i\}$. We assume that $K_{d+3}^d[V(K_{d+3}^d)\backslash \{v_{d+3}\}]$ has already been embedded in $\mathbb{R}^d$. By Lemma \ref{jordan-Brouwer}, each $K_{d+3}^d[V_i]$ will divide $\mathbb{R}^d$ into two disconnected regions, with one of the regions being empty. We designate the non-empty region as the external region of $K_{d+3}^d[V_i]$ and the empty region as the internal region. Let $In(i)$ be the internal region corresponding to $K_{d+3}^d[V_i]$ and $Ex(i)$ be the external region corresponding to $K_{d+3}^d[V_i]$. (As shown in Figure \ref{new7}, when embedded $K_{d+3}^d$ in $\mathbb{R}^d$, it is obvious that line $v_1v_6$ intersects with $K_{d+3}^d[V_i]$ at least at one point. )
	
	Without loss of generality, we assume that $v_{d+3}$ is in $In(1)$, and it follows that $v_1$ is in $Ex(1)$. Since $K_{d+3}^d$ is a complete $d$-uniform-topological hypergraph, there must exist a simplex whose vertex set includes both $v_1$ and $v_{d+3}$. By Lemma \ref{jordan-Brouwer}, this simplex must intersect with $K_{d+3}^d[V_1]$. 
	Therefore $K_{d+3}^d$ cannot be embedded in $\mathbb{R}^d$.
	
\end{proof}

\begin{lemma}\label{bipa}
	$K_{3,d+1}^d$ is a non-$\mathbb{R}^d$-hypergraph. 
\end{lemma}

\begin{proof}
	The proof of Lemma \ref{bipa} is similar to that of Lemma \ref{comp}. 
	(Figure \ref{new8} is an example of $K_{3,4}^3$ in $\mathbb{R}^3$. )

	
\end{proof}

\begin{definition}[{\it anti-$d$-dimension minor}]
	If a $d$-uniform-topological hypergraph $G$ has a $K_{d+3}^d$-minor or $K_{3,d+1}^d$-minor, then we call $G$ an {\it anti-$d$-dimension minor}. 
\end{definition}

\subsection{Proof of Theorem \ref{anti-minor}}\label{last}

In view of Lemma \ref{10.33} and Lemma \ref{10.34}, it suffices to prove Theorem \ref{anti-minor} for triangulated $d$-uniform topological hypergraph whose $1$-skeleton is $(d+1)$-connected. 

\begin{proof}
	
	Let $G$ be a triangulated $d$-uniform topological hypergraph. If its minors include either $K_{3,d+1}^d$ or $K_{d+3}^d$, it is obvious that $G$ is a non-$\mathbb{R}^d$-hypergraph by Lemmas \ref{comp}-\ref{bipa}.

	Now we assume that $G$ is a non-$\mathbb{R}^d$-hypergraph, the $1$-skeleton of $G$ (denoted by $G_{sk}^1$) is $(d+1)$-connected, and $G$ is simple. Because all hypergraphs on $(d+2)$ or fewer vertices can be embedded in $\mathbb{R}^d$, we have $|V(G)|\geq d+3$. We proceed by induction on $|V(G)|$. By Lemma \ref{connected}, $G_{sk}^1$ contains a $1$-dimensional simplex $e = xy$ such that the $1$-skeleton of $H= G/e$ is $(d+1)$-connected. If $H$ is a non-$\mathbb{R}^d$-hypergraph, it has an anti-$d$-dimension minor, by induction. Since every minor of $H$ is also a minor of $G$, we deduce that $G$ too has an anti-$d$-dimension minor. So we may assume that $H$ is an $\mathbb{R}^d$-hypergraph. 
	
	Consider an $\mathbb{R}^d$-embedding $H'$ of $H$. Denote by $z$ the vertex of $H$ formed by contracting $e$. Because $H$ is $(d+1)$-connected, by Lemma \ref{sphere} the neighbors of $z$ lie on a $(d-1)$-sphere $S^{d-1}$, the boundary of some polytope $W$ of $H'- z$. Denote by $B_x$ and $B_y$, respectively, the bridges of $W$ in $G \backslash e$ that contain the vertices $x$ and $y$. 
	
	Note that $B_x$ and $B_y$ cannot avoid each other since $G$ is a triangulated $d$-uniform topological hypergraph. It follows that $B_x$ and $B_y$ overlap. By Lemma \ref{bridges}, they are therefore either skew or else equivalent $(d+1)$-bridges. In the latter case, $G$ has a $K_{d+3}^d$-minor; In the former case, $G$ has a $K_{3,d+1}^d$-minor. 
\end{proof}

Note that in higher dimensions, if $B_x$ and $B_y$ in the above proof are skew, the resulting structure becomes less intuitive. 
We present an example in three-dimensional Euclidean space. It is easy to observe that in this case, $B_x$ and $B_y$ together form a $K_{3,4}^3$, which can be analogized to higher-dimensional cases: 

As shown in Figure \ref{new13}(1), simplices $x_1y_1y_2,x_1y_2y_3,x_1y_1y_3$, $x_2y_1y_2,x_2y_2y_3,x_1y_1y_3$ are joined together to form a two-dimensional sphere $S^2$. Vertices $x$ and $y$ are inside $S^2$. 
As shown in Figure \ref{new13}(2), simplices $xy_1y_2,xy_2y_3$ and $xy_1y_3$ form a bridge $B_1$, which effect a partition of $S^2$ into $2$ disjoint segments. 
As shown in Figure \ref{new13}(3), there is another bridge $B_2$ with internal vertex $y$. Its vertics of attachment includes both $x_1$ and $x_2$. By definition, $B_1$ and $B_2$ are skew. On the other hand, since there are no pendant simplex in the hypergraph, bridge $B_2$ must include simplices $\{yx_1y_1, yx_1y_2, yx_1y_3, yx_2y_1, yx_2y_2, yx_2y_3, yy_1y_2, yy_2y_3, yy_1y_3\}$. 
As shown in Figure \ref{new13}(4), let $A=\{x,x_1,x_2\}$ and $B=\{y,y_1,y_2,y_3\}$, At this point, $K_{3,4}^3\subseteq S^2\cup B_1\cup B_2=(A,B)$ is a complete bipartite $3$-uniform-topological hypergraph.

\begin{figure}
	\centering     
	\includegraphics[width=1\linewidth]{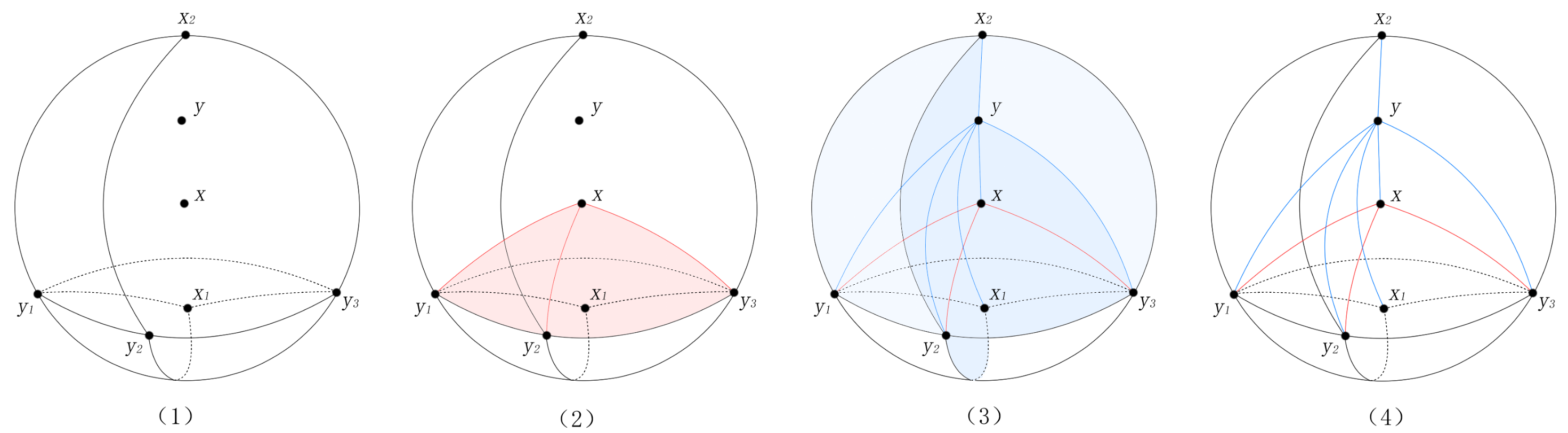}
	\caption{Minimum skew in triangulated $\mathbb{R}^3$-hypergraph ($K_{3,4}^3$-minor).}
	\label{new13}
\end{figure}

\section{Conclusion and outlook}
\label{FW}

The Hadwiger conjecture states that every loopless graph $G$ without a $K_t$ minor satisfies $\chi(G) \leq t-1$. This conjecture has been verified for $1 \leq t \leq 6$. As a generalization of the four color theorem, it remains one of the most significant and challenging open problems in graph theory. Notably, Wagner's theorem and the four color theorem are equivalent to the case $t = 5$ of the Hadwiger Conjecture. 
Given the conjecture's inherent complexity, we attempt to describe the structure of hypergraphs through high-dimensional embeddings, aiming to establish a connection between the chromatic number of a hypergraph and its embeddability in higher-dimensional spaces. 
Crucially, we observe that in Theorem~\ref{anti-minor}, the hypergraph $K_{d+3}^d$ has a $1$-skeleton isomorphic to the complete graph $K_{d+3}$---precisely the structure central to the Hadwiger Conjecture. To leverage high-dimensional embeddings for studying this conjecture, the following objectives must be addressed:
\begin{itemize}
	\item Develop a method for converting graphs into hypergraphs, thereby establishing an embedding theory for graphs in higher-dimensional Euclidean spaces.
	\item Construct a coloring theory for high-dimensional Euclidean spaces, generalizing the four color theorem. Specifically, we aim to prove that if a graph $G$ embeds into $\mathbb{R}^d$, then $\chi(G) \leq d - 1$.
\end{itemize}
Our future research will focus primarily on these directions. Resolving these problems will enable the application of Theorem~\ref{anti-minor} to the Hadwiger Conjecture, offering a fundamentally new framework for its study.

\section*{Acknowledgement}
This work was funded by the National Key R \& D Program of China (No.~2022YFA1005102) and the National Natural Science Foundation of China (Nos.~12325112, 12288101).

\section*{Statements and Declarations}

\begin{itemize}
	\item Competing interests: We hereby declare that there are no conflicts of interest to disclose in relation to this submission. We have no affiliations, relationships, or financial interests that could be perceived as influencing my work.
	\item Data availability: This research does not involve any data generation or analysis. Therefore, no data are available. 
\end{itemize}




\end{document}